\documentclass[matprg]{svjour3}

\usepackage{amsfonts}
\usepackage{latexsym}
\usepackage{amsmath, amssymb, algorithm, algorithmic}

\newcommand{\beq}{\begin{equation}}
\newcommand{\eeq}{\end{equation}}
\newcommand{\beqa}{\begin{eqnarray}}
\newcommand{\eeqa}{\end{eqnarray}}
\newcommand{\beqas}{\begin{eqnarray*}}
\newcommand{\eeqas}{\end{eqnarray*}}
\newcommand{\bi}{\begin{itemize}}
\newcommand{\ei}{\end{itemize}}
\newcommand{\ba}{\begin{array}}
\newcommand{\ea}{\end{array}}

\newcommand{\nn}{\nonumber}

\def\eqnok#1{(\ref{#1})}
\def\argmin{{\rm argmin}}

\def\Argmin{{\rm Argmin}}

\def\vgap{\vspace*{.1in}}

\def\conv{{\rm Conv}}

\def\LMO{{\rm LO}}

\newcommand{\bbe}{\mathbb{E}}

\newcommand{\bbr}{\mathbb{R}}
\def\Tr{{\hbox{\rm Tr}}}

\title{
The Complexity of Large-scale Convex Programming\\
under a Linear Optimization Oracle
}
\author{
    Guanghui Lan
    \thanks{Department of Industrial and Systems
    Engineering, University of Florida, Gainesville, FL, 32611.
    (email: {\tt glan@ise.ufl.edu}).}
}

\begin{document}

\maketitle

\begin{abstract}
This paper considers a general class of iterative optimization algorithms, referred to as 
linear-optimization-based convex programming (LCP) methods, for solving large-scale convex programming (CP) problems.
The LCP methods, covering the classic conditional gradient (CndG) method (a.k.a., Frank-Wolfe method) 
as a special case, can only solve a linear optimization subproblem at each iteration. In this paper,
we first establish a series of lower complexity bounds for the LCP methods to solve
different classes of CP problems, including smooth, nonsmooth
and certain saddle-point problems. We then formally establish the theoretical optimality 
or nearly optimality, in the large-scale case, for the CndG method and its variants to solve different
classes of CP problems. We also introduce several new optimal LCP methods, obtained by
properly modifying Nesterov's accelerated gradient method, and demonstrate their possible advantages
over the classic CndG for solving certain classes of large-scale CP problems.

\vspace{.1in}


\vspace{.07in}


\end{abstract}

\vspace{0.1cm}

\setcounter{equation}{0}
\section{Introduction}
The last few years have seen an increasing interest in the application
of convex programming (CP) models for machine learning, image processing, and 
polynomial optimization, etc. The CP problems arising from these applications, however, are often
of high dimension and hence challenging to solve.
In particular, they are generally beyond the capability of 
second-order interior-point methods due to the highly demanding iteration costs
of these optimization techniques. This has motivated the currently active research on first-order methods
which possess cheaper iteration costs for large-scale CP, including Nesterov's optimal method~\cite{Nest83-1,Nest04,Nest05-1}
and several stochastic first-order algorithms in~\cite{NJLS09-1,Lan10-3}.
These optimization algorithms are relatively simple, and suitable for the situation when low or moderate solution accuracy is sought-after. 

In this paper, we study a different class of optimization algorithms, referred to as
{\sl linear-optimization-based convex programming (LCP)} methods, for large-scale CP. Specifically, consider the CP problem of
\beq \label{cp}
f^* := \min_{x \in X} f(x),
\eeq
where $X \subseteq \bbr^n$
is a convex compact set and $f: X \to \bbr$ is a closed convex function.
The LCP methods solve problem \eqnok{cp} by iteratively 
calling a {\sl linear optimization} ($\LMO$) oracle, which, for a given input
vector $p \in \bbr^n$, computes the solution of subproblems given in the form of
\beq \label{CG_subproblem}
\Argmin_{x \in X} \langle p, x \rangle.
\eeq
In particular, if $p$ is computed
based on first-order information, then we call these algorithms {\sl first-order LCP} methods. 
Clearly, the difference between first-order LCP methods and the more general first-order methods
exists in the restrictions on the format of subproblems. 
For example, in the well-known subgradient (mirror) descent method \cite{nemyud:83}
and Nesterov's method~\cite{Nest83-1,Nest04}, we solve the projection (or prox-mapping)
subproblems given in the form of
\beq \label{FO_subproblem}
\argmin_{x \in X} \left\{\langle p, x \rangle + d(x) \right\}.
\eeq
Here $d: X \to \bbr$ is a certain strongly convex function (e.g., $d(x) = \|x\|_2^2/2$). 

The development of LCP methods dates back to the conditional gradient (CndG) method (a.k.a.,
Frank-Wolfe algorithm) developed by~\cite{FrankWolfe56-1} (see also
\cite{Dunn79,Dunn80} for some earlier studies on this area).
This method has recently regained some interests from both 
machine learning and optimization community (see, e.g.,~\cite{AhiTodd13-1,Bach12-1,BeckTeb04-1,CoxJudNem13-1,Clarkson10,Freund13,Hazan08,HarJudNem12-1,Jaggi11,Jaggi13,Jaggi10,LussTeb13-1,ShGosh11,ShenKim12-1})
mainly for the following reasons.
\begin{itemize}
\item {\sl Low iteration cost.} In many cases, the solution of the linear subproblem
\eqnok{CG_subproblem} is much easier to solve than the nonlinear subproblem \eqnok{FO_subproblem}.
For example, if $X$ is a spectrahedron given by
$X=\{x \in \bbr^{n \times n}: \Tr(x) = 1, x \succeq 0\}$, the solution of \eqnok{CG_subproblem}
can be much faster than that of \eqnok{FO_subproblem}.
\item {\sl Simplicity.} The CndG method is simple to implement since it does not require 
the selection of the distance function $d(x)$ in \eqnok{FO_subproblem} and the fine-tuning of 
stepsizes, which are required in most other first-order methods (with exceptions to some extent for a few level-type 
first-order methods, see~\cite{BenNem05-1,Lan13-1}). This property is
also referred to affine invariance (see \cite{Jaggi13,AspJaggi13} for more discussions).
\item {\sl Structural properties for the generated solutions.} The output solutions of the CndG method may have 
certain desirable structural properties, e.g., sparsity and low rank, as 
they can often be written as the convex combination of a small number of extreme points of $X$.
\end{itemize} 
Numerical studies (e.g., \cite{HarJudNem12-1}) indicate that the CndG method can be competitive to 
the more involved gradient-type methods for solving certain classes of CP problems. 
It is also worth noting that the CndG method is closely related to the von Neumann algorithm studied
by Dantzig~\cite{Dantzig91-1,Dantzig92-1}, and later in Epelman and Freund \cite{EpelFreu00-1},
which can be viewed as a specialized CndG method for solving linear/conic feasibility problems.

This paper focuses on the complexity analysis of CP under an $\LMO$ oracle,
as well as the development of new LCP methods for large-scale CP. 
In particular, we intend to provide a general framework for complexity studies
for the LCP methods, by generalizing a few interesting complexity results existing 
in the literature (e.g., \cite{Clarkson10,HarJudNem12-1,Hazan08,Jaggi10,Jaggi13}).
Although
there exists rich complexity theory for the general first-order methods
for large-scale CP in the literature, the study on the complexity of CP under an $\LMO$ oracle
is still limited. More specifically, in view of the classic CP complexity theory~\cite{nemyud:83,Nest04},
if $f$ is a general nonsmooth Lipschitz continuous convex function such that 
\beq \label{nonsmooth}
|f(x) - f(y)| \le M \|x - y\|, \, \, \forall x, y \in X,
\eeq
then the number of iterations required by any first-order methods
to find an $\epsilon$-solution of \eqnok{cp}, i.e.,
a point $\bar x \in X$ s.t. $f(\bar x) - f^* \le \epsilon$,
cannot be smaller than ${\cal O}(1/\epsilon^2)$ if $n$ is sufficiently large. 
In addition, if $f$ is a general smooth convex function satisfying
\beq \label{smooth}
\|f'(x) - f'(y)\|_* \le L \|x- y\|, \forall x, y \in X,
\eeq
then the number of iterations required by any first-order methods
to find an $\epsilon$-solution of \eqnok{cp} cannot be smaller than ${\cal O}(1/\sqrt{\epsilon})$
if $n$ is large enough. These lower complexity bounds can be achieved, for example, by the 
aforementioned subgradient (mirror) descend method
and Nesterov's method, respectively, for nonsmooth and smooth convex optimization. 
In addition, in a breakthrough paper, ~\cite{Nest05-1} studied 
an important class of saddle point problems with $f$ is given by
\beq \label{saddleclass}
f(x) = \max_{y \in Y} \left\{ \langle Ax, y \rangle - \hat f(y)\right\}.
\eeq
Here $Y \subseteq \bbr^m$ is a convex compact set, $A: \bbr^n \to \bbr^m$
a linear operator and $\hat f: Y \to \bbr$ is a simple convex function. Although
$f$ given by \eqnok{saddleclass} is nonsmooth in general, Nesterov showed that
it can be closely approximated by a smooth function. Accordingly, he devised a novel 
smoothing scheme that can achieve the ${\cal O}(1/\epsilon)$ for solving this class of saddle point problems.
It should be noted, however, that the complexity of these problems, in
terms of the number of calls to the $\LMO$ oracle, does not seem to be fully understood yet.

Our contribution in this paper lies on the following three aspects.
Firstly, we establish a series of lower complexity bounds for solving different classes of CP problems
under an $\LMO$ oracle. In particular, we show that for 
solving general smooth CP problems satisfying \eqnok{smooth}, the complexity 
(or number of calls to the $\LMO$ oracle), in the worst case, cannot be smaller than
\beq \label{smooth_lb}
{\cal O}(1) \min \left\{ n, \frac{L D_X^2}{\epsilon}\right\},
\eeq
where ${\cal O}(1)$ denotes an absolute constant, $n$ is the dimension of the problem, and $D_X:= \max_{x, y\in X} \|x-y\|$. 
It is worth noting that a similar lower bound has been established 
for the CndG method by~\cite{Jaggi13}. However, the lower bound
in \eqnok{smooth_lb} shows explicitly the dependence on the dimension
$n$, and the problem parameters $L$ and $D_X$.
Moreover, for solving the aforementioned saddle point problems with $f$ given by 
\eqnok{saddleclass}, we show that the number of calls to the $\LMO$ oracle cannot be
smaller than 
\beq \label{sd_lb}
{\cal O}(1) \min\left \{n, \frac{\|A\|^2 D_X^2 D_Y^2}{\epsilon^2}\right\}.
\eeq
We further show that the number of calls to the
$\LMO$ oracle for solving general nonsmooth CP problems cannot be smaller than 
\beq \label{nonsmooth_lb}
{\cal O} (1) \min \left\{n, \frac{M^2 D_X^2}{\epsilon^2} \right\}.
\eeq
It should be pointed out that these lower complexity bounds are obtained not only for
the aforementioned first-order LCP methods, but also for any other LCP methods
including those based on higher-order information.

Secondly, we formally establish the (near) optimality of the CndG method and its variants,
in terms of the number of calls to the $\LMO$ oracle,
for solving different classes of CP problems under an $\LMO$ oracle. 
\begin{enumerate}
\item [a)] If $f$ is a smooth convex function satisfying
\eqnok{smooth}, it is well-known that the number of iterations required by the classic
CndG method to find an $\epsilon$-solution of \eqnok{cp}
will be bounded by ${\cal O} (1/\epsilon)$ (see, e.g., \cite{Jaggi11,HarJudNem12-1,Jaggi13}). Hence, in view of \eqnok{smooth_lb},
the classic CndG is an optimal LCP method if $n$ is sufficiently large, i.e., $n \ge L D_X^2/\epsilon$. 
Moreover, it is also well-known that for general first-order methods, one can employ non-Euclidean
norm $\|\cdot\|$ and the distance function $d(x)$ in \eqnok{FO_subproblem}
to accelerate the solutions for CP problems with certain types of feasible sets $X$. 
However, the CndG method is invariant to the selection of $\|\cdot\|$ and thus
self-adaptive to the geometry of the feasible region $X$ (see also \cite{Jaggi11,Jaggi13}). 
\item [b)] If $f$ is a special nonsmooth function given by \eqnok{saddleclass},
we show that the CndG method can achieve the lower complexity bound in \eqnok{sd_lb}
after properly smoothing the objective function.
Note that, although a similar bound has been developed in \cite{CoxJudNem13-1}, 
the optimality of this bound has not yet been established. In addition, the smoothing technique developed
here is slightly different from those in \cite{Nest05-1,CoxJudNem13-1} as we do not require
explicit knowledge of $D_X$, $D_Y$ and the target accuracy $\epsilon$ given in 
advance. 
\item [c)] If $f$ is a general nonsmooth function satisfying \eqnok{nonsmooth},
we show that the CndG method can achieve a nearly optimal complexity bound in terms
of its dependence on $\epsilon$ after properly incorporating the randomized 
smoothing technique (e.g., \cite{DuBaMaWa11}). In particular,
by applying this method to the bilinear saddle point problems with $f$
given by \eqnok{saddleclass}, we obtain an first-order algorithm
which only requires linear optimization in both primal and dual space
to solve this class of problems. It appears to us that no such techniques have been presented before
in the literature (see discussions in Section 1 of \cite{Nest08-1}).
\item [d)] We also discuss the possibility to improve the complexity of 
the CndG method under strong convexity assumption about $f(\cdot)$ and 
with an enhanced $\LMO$ oracle (see also a related work by~\cite{GarberHazan13}).
\end{enumerate}

Thirdly, we present a few new LCP methods, namely
the primal averaging CndG (PA-CndG) and primal-dual averaging CndG (PDA-CndG) algorithms, 
for solving large-scale CP problems under an $\LMO$ oracle. 
These methods are obtained by replacing the projection subproblems with linear optimization subproblems
in Nesterov's accelerated gradient methods.
We demonstrate that these new LCP methods not only exhibit the aforementioned optimal
(or nearly optimal) complexity bounds, in terms of the number of calls to the $\LMO$ oracle,
for solving different CP problems, but also possess some unique convergence properties. In particular, we show that the rate of convergence of these new LCP methods depends
on the summation of the distances among the solutions of \eqnok{CG_subproblem}.
By exploiting this fact, we develop certain necessary conditions
for the $\LMO$ oracle under which the PA-CndG and PDA-CndG would exhibit an ${\cal O}(1/\sqrt{\epsilon})$
iteration complexity for solving smooth CP problems. This result thus helps to build up the connection between
LCP methods and the general optimal first-order methods 
for CP. We also demonstrate through our preliminary numerical experiments that
one of these new methods, namely PDA-CndG, can significantly outperform the
CndG method for solving certain classes of CP problems, e.g., those with
box-type constraints. 

This paper is organized as follows. We first introduce a few lower complexity bounds
for solving different classes of CP problems under an $\LMO$ oracle in Section 2. In Section 3,
we formally show the optimality of the classic CndG method for solving smooth CP problems,
develop different variants of the CndG method which are optimal or nearly optimal
for solving different nonsmooth CP problems, and present possible improvement of
the CndG method to solve strongly convex CP problems. We then present a few new 
LCP methods, namely PA-CndG and PDA-CndG, establish their convergence properties and
conduct numerical comparisons in Section 4. Some brief concluding remarks are
made in Section 5.

\subsection{Notation and terminology}
Let $X \in \bbr^n$ and $Y \in \bbr^m$ be given convex compact sets.
Also let $\|\cdot\|_X$ and $\|\cdot\|_Y$ be the norms (not necessarily
associated with inner product) in $\bbr^n$ and $\bbr^m$,
respectively. For the sake of simplicity, we often skip the subscripts in the norms $\|\cdot\|_X$
and $\|\cdot\|_Y$. We define the diameter of the sets $X$ and $Y$, respectively, as
\beq \label{def_DX}
D_X \equiv D_{X,\|\cdot\|}:= \max_{x, y \in X} \|x - y\|
\eeq
and
\beq \label{def_DY}
D_Y \equiv D_{Y,\|\cdot\|} := \max_{x, y \in Y} \|x - y\|.
\eeq
For a given norm $\|\cdot\|$, we denote its conjugate by
$\|s\|_* = \max_{\|x\| \le 1} \langle s, x \rangle$. We use $\|\cdot\|_1$
and $\|\cdot\|_2$, respectively, to denote the regular $l_1$ and $l_2$ norms.
Let $A: \bbr^n \to \bbr^m$ be a given linear operator, we use $\|A\|$
to denote its operator norm given by $\|A\| := \max_{\|x\| \le 1} \|Ax\|$.
Let $f: X \to \bbr$ be a convex function, we
denote its linear approximation at $x$ by
\beq \label{def_lf}
l_f(x;y) := f(x) + \langle f'(x), y - x \rangle.
\eeq
Clearly, if $f$ satisfies \eqnok{smooth}, then
\beq \label{smoothness}
f(y) \le l_f(x;y) + \frac{L}{2} \|y-x\|^2, \ \ \
\forall \, x, y \in X.
\eeq
Notice that the constant $L$ in \eqnok{smooth} and \eqnok{smoothness}
depends on $\|\cdot\|$.

\setcounter{equation}{0}
\section{Lower Complexity Bounds for CP under an $\LMO$ oracle}
Our goal in this section is to establish a few lower complexity bounds 
for solving different classes of
CP problems under an $\LMO$ oracle. More specifically,
we first introduce a generic LCP algorithm in Subsection~\ref{sec_generic} and then present a few 
lower complexity bounds for these types of algorithms to solve 
different smooth and nonsmooth CP problems
in Subsections~\ref{sec_smooth} and \ref{sec_nonsmooth}, respectively.

\subsection{A generic LCP algorithm} \label{sec_generic}
The LCP algorithms solve problem \eqnok{cp} iteratively. In particular,
at the $k$-th iteration, these algorithms perform a call to the $\LMO$ oracle in order to update the iterates by
minimizing a given linear function $\langle p_k, x \rangle$ over the feasible region $X$.
A generic framework for these types of algorithms is described as follows.

\begin{algorithm} [H]
	\caption{A generic LCP algorithm}
	\label{algGeneric}
	\begin{algorithmic}
\STATE Let $x_0 \in X$ be given.

\FOR {$k=1, 2, \ldots,$ }
\STATE Define the linear function $\langle p_k, \cdot \rangle$.
\STATE Call the $\LMO$ oracle to compute $x_{k} \in \Argmin_{x\in X}\langle p_k, x \rangle$. 
\STATE Output $y_k \in \conv\{x_0, \ldots, x_{k}\}$.
\ENDFOR
	\end{algorithmic}
\end{algorithm}

Observe the above LCP algorithm can be quite general. 
Firstly, there are no restrictions regarding the definition of the linear function $\langle p_k, \cdot\rangle$. 
For example, if $f$ is a smooth function, then $p_k$ can be defined as the gradient computed at
some feasible solution or a linear combination of some previously
computed gradients. If $f$ is nonsmooth, we can define $p_k$ as the gradient computed 
for a certain approximation function of $f$. We can also consider 
the situation when some random noise or 
second-order information is incorporated into the definition of $p_k$.
Secondly, the output solution $y_k$ 
is written as a convex combination of $x_0, \ldots, x_{k}$, and thus can 
be different from any points in $\{x_k\}$.
We will show in Sections~\ref{sec_ub} and~\ref{sec_acLCP} that Algorithm~\ref{algGeneric} covers, as certain special cases,
the classic CndG method and several new LCP methods to be studied in this paper.

It is interesting to observe the difference between the above LCP algorithm
and the general first-order methods for CP. One one hand, the LCP algorithm 
can only solve linear, rather than nonlinear subproblems (e.g., projection or
prox-mapping) to update iterates. On the other hand, the LCP algorithm allows 
more flexibility in the definitions of the search direction $p_k$ and the output solution
$y_k$. 

\subsection{Lower complexity bounds for smooth minimization}~\label{sec_smooth}
In this subsection, we consider a class of smooth CP problems, denoted by ${\cal F}^{1,1}_{L,\|\cdot\|}(X)$,
which consist of any CP problems given in the form of \eqnok{cp} with $f$
satisfying assumption \eqnok{smooth}. Our goal is to derive a lower bound
on the number of iterations required by any LCP methods 
for solving this class of problems.

The complexity analysis has been an important topic in convex programming (see~\cite{nemyud:83,Nest04}). However, the study on the complexity for 
LCP methods is quite limited. Existing results focus on a specific algorithm, 
namely the classic CndG method. More specifically, \cite{CanonCullum68-1} 
proved an asymptotic lower bound of $\Omega(1/k^{1+\mu})$, for any $\mu > 0$,
on the rate of convergence for the CndG method. \cite{Jaggi13} revisited this algorithm 
and established a lower bound on the number of iteration performed by this algorithm
for finding an approximate solution with certain sparse pattern (see also \cite{Clarkson10} and \cite{Hazan08}).
Using the basic idea of Jaggi's 
development, we provide lower complexity bounds which has explicit dependence on
the problems dimension and other parameters, in addition to the target accuracy. Moreover,
while the lower bounds in \cite{Jaggi13} were developed for smooth optimization problem,
we also generalize these bound for nonsmooth and saddle point problems.
for solving different classes of CP problem. 

Similarly to the classic complexity analysis for CP in
\cite{nemyud:83,Nest04}, we assume that the $\LMO$ oracle used
in the LCP algorithm is {\sl resisting},
implying that: i) the LCP algorithm does not know how the solution of
\eqnok{CG_subproblem} is computed; and ii) in the worst case, the
$\LMO$ oracle provides the least amount of information for the
LCP algorithm to solve problem \eqnok{cp}. 
Using this assumption, we will construct a class of worst-case instances in ${\cal F}^{1,1}_{L,\|\cdot\|}(X)$,
inspired by \cite{Jaggi13},
and establish a lower bound on the number of iterations required by any LCP algorithms
to solve these instances. 











\begin{theorem} \label{the_lb}
Let $\epsilon > 0$ be a given target accuracy. 
The number of iterations required by any LCP methods to solve the problem class ${\cal F}_{L,\|\cdot\|}^{1,1}(X)$, in the worst case,
cannot be smaller than
\beq \label{lowerbound}
\left\lceil \min\left\{ \frac{n}{2}  , \frac{L D_X^2}{4 \epsilon} \right \} \right\rceil - 1,
\eeq 
where $D_X$ is given by \eqnok{def_DX}.
\end{theorem}

\begin{proof}
Consider the CP problem of 
\beq \label{lb_problem}
f_0^* := \min_{x \in X_0}  \left\{ f_0(x) := \frac{L}{2} \sum_{i=1}^n (x^{(i)})^2 \right\},
\eeq
where $X_0 := \left\{x  \in \bbr^n: \sum_{i=1}^n x^{(i)} = D, x^{(i)} \ge 0\right\}$ for some
$D > 0$. It can be easily seen that the optimal solution $x^*$ and the optimal value 
$f^*_0$ for problem \eqnok{lb_problem} are given by
\beq \label{lb_opt_val}
x^* = \left(\frac{D}{n}, \ldots, \frac{D}{n} \right) \ \ \mbox{and} \ \
f^*_0 = \frac{L D^2}{n}. 
\eeq
Clearly, this class of problems belong to ${\cal F}_{L, \|\cdot\|}^{1,1}(X)$ with
$\|\cdot\| = \|\cdot\|_2$.

Without loss of generality, we assume that the initial point is given by $x_0 = D e_1$
where $e_1 = (1, 0, \ldots, 0)$ is the unit vector. Otherwise, for an arbitrary $x_0 \in X_0$, we can 
consider a similar problem given by
\[
\begin{array}{ll}
\min_{x} & \left( x^{(1)}\right)^2 + \sum_{i=2}^n \left( x^{(i)} - x_0^{(i)}\right)^2\\
\mbox{s.t.} & x^{(1)} + \sum_{i=2}^n \left(x^{(i)} - x_0^{(i)}\right) = D\\
& x^{(1)} \ge 0 \\
& x^{(i)} - x_0^{(i)} \ge 0, i = 2, \ldots, n.
\end{array}
\]
and adapt our following argument to this problem without much modification.

Now suppose that problem~\eqnok{lb_problem} is to be solved
by an LCP algorithm. At the $k$-th iteration, this algorithm will call
the $\LMO$ oracle to compute a new search point $x_k$ based on the input vector $p_k$, $k = 1, \ldots$.
We assume that the $\LMO$ oracle is resisting in the sense that it always outputs an extreme point 
$x_k \in \{ D e_1, D e_2, \ldots, D e_n \}$ such that
\[
x_k \in \Argmin_{x \in X_0} \langle p_k, x \rangle.
\]
Here $e_i$, $i = 1, \ldots, n$, denotes the $i$-th unit vector in $\bbr^n$.
In addition, whenever $x_k$ is not uniquely defined, it breaks the tie arbitrarily.
Let us denote $x_k = D e_{p_k}$ for some $1 \le p_k \le n$.
By definition, we have $y_k \in D \, \conv\{x_0, x_1, \ldots, x_k\}$ and hence
\beq \label{output_conv}
y_k \in D \, \conv\{e_1, e_{p_1}, e_{p_2}, \ldots, e_{p_k} \}.
\eeq
Suppose that totally $q$ unit vectors
from the set $\{e_1,
e_{p_1}, e_{p_2}, \ldots, e_{p_k}\}$ are linearly independent
for some $1 \le q \le k+1 \le n$.
Without loss of generality, assume that the vectors $e_1$, $e_{p_1}$, $e_{p_2}$, \ldots, $e_{p_{q-1}}$
are linearly independent. Therefore, we have
\beqas
f_0(y_k) &\ge& \min_{x} \left\{f_0(x): x \in D \, \conv\{e_1, e_{p_1}, e_{p_2}, \ldots, e_{p_k} \}\right\} \\
&=& \min_{x} \left\{f_0(x): x \in D \, \conv\{e_1, e_{p_1}, e_{p_2}, \ldots, e_{p_{q-1}} \}\right\}\\
&=&  \frac{L D^2}{q} \ge \frac{L D^2}{k+1},
\eeqas
where the second identity follows from the definition of $f_0$ in \eqnok{lb_problem}.
The above inequality together with \eqnok{lb_opt_val} then imply that
\beq \label{lb_ineq}
f_0(y_k) - f_0^* \ge \frac{L D^2}{k+1} - \frac{L D^2}{n}
\eeq
for any $k = 1, \ldots, n-1$. Let us denote
\[
\bar K := \left\lceil \min\left\{ \frac{n}{2}  , \frac{L D_{X_0}^2}{4 \epsilon} \right \} \right\rceil - 1.
\]
By the definition of $D_X$ and $X_0$, and the fact that $\|\cdot\| = \|\cdot\|_2$, we can easily see that 
$D_{X_0} = \sqrt{2} D$ and hence that
\[
\bar K = \left\lceil \frac{1}{2} \min\left\{ n  , \frac{L D^2}{\epsilon} \right \} \right\rceil - 1.
\]
Using \eqnok{lb_ineq} and the above identity, we conclude that, for any $1 \le k \le \bar K$,
\beqas
f_0(y_k) - f_0^* &\ge& \frac{L D^2}{\bar K + 1} - \frac{L D^2}{n} \ge \frac{2 L D^2}{\min\left\{ n  , \frac{L D^2}{\epsilon} \right \}}
- \frac{L D^2}{n}\\
&=& \frac{L D^2}{\min\left\{ n  , \frac{L D^2}{\epsilon} \right \}} + \left(\frac{L D^2}{\min\left\{ n  , \frac{L D^2}{\epsilon} \right \}}
- \frac{L D^2}{n}\right) \ge \frac{L D^2}{\frac{L D^2}{\epsilon}} + \left(\frac{L D^2}{n }
- \frac{L D^2}{n}\right) = \epsilon.
\eeqas
Our result then immediately follows
since \eqnok{lb_problem} is a special class of problems in ${\cal F}_{L,\|\cdot\|}^{1,1}(X)$.
\end{proof}

\vgap

We now add a few remarks about the results obtained in Theorem~\ref{the_lb}.
First, it can be easily seen from \eqnok{lowerbound} that, if $n \ge L D_X^2 / (2\epsilon)$,
then the number of iterations required by any LCP methods
for solving ${\cal F}^{1,1}_{L,\|\cdot\|}(X)$, in the worst case, 
cannot be smaller than ${\cal O}(1) L D_X^2/\epsilon$. Second, it is worth noting that the 
objective function $f_0$ in \eqnok{lb_problem}
is actually strongly convex. Hence, the performance of the LCP methods,
in terms of the number of calls to the $\LMO$ oracle,
cannot be improved by assuming strong convexity when $n$ is sufficiently large
(see Section~\ref{sec_opt_strong} for more discussions). 

\subsection{Lower complexity bounds for nonsmooth minimization}~\label{sec_nonsmooth}
In this subsection, we consider two classes of nonsmooth CP problems.
The first one is a general class of nonsmooth CP problems, denoted by  ${\cal F}^{0}_{M,\|\cdot\|}(X)$,
which consist of any CP problems given in the form of \eqnok{cp} 
with $f$ satisfying \eqnok{nonsmooth}.
The second one is a special class of bilinear saddle-point problems,
denoted by  ${\cal F}^{0}_{\|A\|}(X, Y)$, composed of all CP problems \eqnok{cp} with $f$ given by \eqnok{saddleclass}. 
Our goal in this subsection is to derive the lower complexity bounds for any LCP algorithms to 
solve these two classes of nonsmooth CP problems.

It can be seen that, if $f(\cdot)$ is given by \eqnok{saddleclass}, then
\[
\|f'(x)\|_* \le \|A\| D_Y, \ \ \forall x \in X,
\]
where $D_Y$ is given by \eqnok{def_DX}. Hence, the saddle point 
problems ${\cal F}^{0}_{\|A\|}(X, Y)$ are a special class of
nonsmooth CP problems.

\vgap

Theorem~\ref{the_lb_nonsmooth} below provides a few lower complexity bounds for solving
these two classes of nonsmooth CP problems by using LCP algorithms.

\begin{theorem} \label{the_lb_nonsmooth}
Let $\epsilon > 0$ be a given target accuracy. Then, the
number of iterations required by any LCP methods to solve 
the general nonsmooth problems  ${\cal F}^{0}_{M,\|\cdot\|}(X)$ and 
saddle point problems ${\cal F}^{0}_{\|A\|}(X, Y)$,
respectively, cannot be smaller than
\beq \label{lb_nonsmooth}
\frac{1}{4} \min\left\{n, \frac{M^2 D_X^2}{2 \epsilon^2} \right\} - 1
\eeq
and 
\beq \label{lb_saddle}
\frac{1}{4} \min\left\{n, \frac{\|A\|^2 D_X^2 D_Y^2}{2 \epsilon^2} \right\} - 1,
\eeq
where $D_X$ and $D_Y$ are defined in \eqnok{def_DX} and \eqnok{def_DY}, respectively.
\end{theorem}

\begin{proof}
We first show the bound in \eqnok{lb_nonsmooth}. Consider the CP problem of
\beq \label{lb_problem1}
\hat{f}_0^* := \min_{x \in X_0} \left\{ \hat{f}(x) := M \left( \sum_{i=1}^n x_i^2 \right)^\frac{1}{2} \right\},
\eeq
where $X_0 := \left\{x  \in \bbr^n: \sum_{i=1}^n x^{(i)} = D, x^{(i)} \ge 0\right\}$ for some
$D > 0$. It can be easily seen that the optimal solution $x^*$ and the optimal value 
$f^*_0$ for problem \eqnok{lb_problem1} are given by
\beq \label{lb_opt_val1}
x^* = \left(\frac{D}{n}, \ldots, \frac{D}{n} \right) \ \ \mbox{and} \ \
\hat f^*_0 = \frac{M D}{\sqrt{n}}. 
\eeq
Clearly, this class of problems belong to ${\cal F}^{0}_{M,\|\cdot\|}(X)$ with $\|\cdot\| = \|\cdot\|_2$.
Now suppose that problem \eqnok{lb_problem} is to be solved by an arbitrary LCP method. 
Without loss of generality, we assume that the initial point is given by $x_0 = D e_1$
where $e_1 = (1, 0, \ldots, 0)$ is the unit vector.  Assume that the $\LMO$ oracle is 
resisting in the sense that it always outputs an extreme point 
solution. By using an argument similar to the one used in the proof of \eqnok{output_conv}, 
we can show that 
\[
y_k \in D \, \conv\{e_1, e_{p_1}, e_{p_2}, \ldots, e_{p_k} \}
\]
where $e_{p_i}$, $i = 1, \ldots, k$, are the unit vectors in $\bbr^n$.
Suppose that totally $q$ unit vectors in the set $\{e_1,
e_{p_1}, e_{p_2}, \ldots, e_{p_k}\}$ are linearly independent
for some $1 \le q \le k+1 \le n$.
We have
\[
\hat f_0(y_k) \ge \min_{x} \left\{\hat f_0(x): x \in D \, \conv\{e_1, e_{p_1}, e_{p_2}, \ldots, e_{p_k} \}\right\} 
=  \frac{M D}{\sqrt{q}} \ge \frac{M D}{\sqrt{k+1}},
\]
where the identity follows from the definition of $\hat f_0$ in \eqnok{lb_problem1}.
The above inequality together with \eqnok{lb_opt_val1} then imply that
\beq \label{lb_ineq1}
\hat f_0(y_k) - \hat f_0^* \ge \frac{M D}{\sqrt{k+1}} - \frac{L D^2}{\sqrt{n}}
\eeq
for any $k = 1, \ldots, n-1$. Let us denote
\[
\bar K := \frac{1}{4}\left\lceil \min\left\{ n, \frac{M^2 D_{X_0}^2}{2 \epsilon^2} \right \} \right\rceil - 1.
\]
Using the above definition, \eqnok{lb_ineq1} and the fact that
$D_{X_0} = \sqrt{2} D$, we conclude that
\[
\hat f_0(y_k) - \hat f_0^* \ge \frac{M D}{\sqrt{\bar K + 1}} - \frac{M D}{n} \ge \frac{2 M D}{\min\left\{ \sqrt{n}  , 
\frac{M D}{\epsilon} \right \}}
- \frac{M D}{\sqrt{n}} \ge \epsilon
\]
for any $1 \le k \le \bar K$.
Our result in \eqnok{lb_nonsmooth} then immediately follows since \eqnok{lb_problem1} is a special class of problems in ${\cal F}_{M,\|\cdot\|}^{0}(X)$.

In order to prove the lower complexity bound in \eqnok{lb_saddle}, 
we consider a class of saddle point problems given in the form of
\beq \label{lb_problem2}
\min_{x \in X_0} \max_{\|y\|_2 \le \tilde D} M \langle x, y \rangle. 
\eeq
Clearly, these problems belong to ${\cal S}_{\|A\|}(X,Y)$ with $A = M I$.
Noting that problem~\eqnok{lb_problem2} is equivalent to
\[
\min_{x \in X_0} M \tilde D \left( \sum_{i=1}^n x_i^2 \right)^\frac{1}{2},
\]
we can show the lower complexity bound in \eqnok{lb_saddle} by using an argument similar 
to the one used in the proof of bound \eqnok{lb_nonsmooth}.
\end{proof}

\vgap

Observe that while the lower complexity bound in \eqnok{lb_nonsmooth}
is in the same order of magnitude as the one established in \cite{nemyud:83} (see also \cite{Nest04})
for the general first-order methods to solve ${\cal F}_{M,\|\cdot\|}^{0}(X)$. However,
the bound in \eqnok{lb_nonsmooth} holds not only for first-order LCP methods, but also 
for any other LCP methods, including those based
on higher-order information to solve ${\cal F}_{M,\|\cdot\|}^{0}(X)$. 

\setcounter{equation}{0}
\section{The Optimality of CndG Methods for CP under an $\LMO$ oracle} \label{sec_ub}
Our goal in this section is to establish the optimality
or near optimality of the classic CndG method and its variants for solving 
different classes of CP problems under an $\LMO$ oracle.
More specifically, we discuss the classic CndG method for solving
smooth CP problems ${\cal F}^{1,1}_{L,\|\cdot\|}(X)$ in Subsection~\ref{sec_classicCG},
and then present different variants of the
CndG method to solve general nonsmooth CP problems  
${\cal F}^{0}_{\|A\|}(X, Y)$ and saddle point problems ${\cal F}^{0}_{M,\|\cdot\|}(X)$,
respectively, in Subsections~\ref{sec_opt_sad} and \ref{sec_opt_nonsm}.
Some discussions about strongly convex problems are included in Subsection~\ref{sec_opt_strong}.  


\subsection{Optimal CndG methods for smooth problems under an $\LMO$ oracle} \label{sec_classicCG}
The classic CndG method~\cite{FrankWolfe56-1,DemRub70} is one of the earliest iterative algorithms to
solve problem~\eqnok{cp}. The basic scheme of this algorithm is stated as follows.

\begin{algorithm} [H]
	\caption{The Classic Conditional Gradient (CndG) Method}
	\label{algFW}
	\begin{algorithmic}
\STATE Let $x_0\in X$ be given. Set $y_0 = x_0$.

\FOR {$k=1, \ldots$ }
\STATE Call the $\LMO$ oracle to compute $x_{k} \in \Argmin_{x\in X}\langle f'(y_{k-1}), x \rangle$. 
\STATE Set $y_{k}  = (1 - \alpha_{k}) y_{k-1} + \alpha_{k} x_{k}$ for some $\alpha_k \in [0,1]$.
\ENDFOR
	\end{algorithmic}
\end{algorithm}

We now add a few remarks about the classic CndG method. Firstly, it can be easily seen that the classic 
CndG method is a special case of the LCP algorithm discussed in Subsection~\ref{sec_generic}.
More specifically, the search direction $p_k$ appearing in the generic LCP algorithm
is simply set to the gradient $f'(y_{k-1})$ in Algorithm~\ref{algFW}, and the output $y_k$ is 
taken as a convex combination of $y_{k-1}$ and $x_k$. Secondly, in order to guarantee the 
convergence of the classic CndG method, we need to
properly specify the stepsizes $\alpha_k$ used in the definition of $y_k$.
There are two popular options for selecting $\alpha_k$: one is to set 
\beq \label{FW_step1}
\alpha_k = \frac{2}{k+1}, \ \ \ k = 1, 2, \ldots,
\eeq
and the other is to compute $\alpha_k$ by solving a one-dimensional minimization problem:
\beq \label{FW_step2}
\alpha_k = \argmin_{\alpha \in [0,1]} f((1-\alpha) y_{k-1} + \alpha x_{k} ), \ \ \ k = 1, 2, \ldots.
\eeq
It is well-known that if $f$ satisfies \eqnok{smooth} and $\alpha_k$ is set to either \eqnok{FW_step1} or \eqnok{FW_step2}, 
then the classic CndG method will exhibit an ${\cal O}(1/k)$ rate of convergence for solving problem \eqnok{cp} (see, e.g., 
\cite{Clarkson10,HarJudNem12-1,Hazan08,Jaggi10,Jaggi11,Jaggi13}).

\vgap

We now formally describe the convergence properties of the above classic CndG method. 
Observe that, in contrast with existing analysis of the classic CndG method, we 
state explicitly in Theorem~\ref{Theorem_FW} how the rate of convergence associated with
this algorithm depends on distance between the previous iterate $y_{k-1}$ and
the output of the $\LMO$ oracle, i.e., $\|x_k - y_{k-1}\|$. In addition, our analysis 
for the classic CndG method is slightly different than the standard ones, and
some of the techniques developed here will be used later for the analysis of some new LCP methods
in Section~\ref{sec_acLCP}. Also observe that, given a candidate solution $\bar x \in X$,
we use the functional optimality gap $f(\bar x) - f^*$ as a termination criterion for the
algorithm. It is worth noting that~\cite{Jaggi13} has recently showed that
the CndG method also exhibit ${\cal O}(1/k)$ rate of convergence in terms of
a stronger termination criterion, i.e.,
the Wolfe gap given by $\max_{x \in X} \langle f'(\bar x), \bar x - x\rangle$,
although \cite{Kha96} implicitly derived such bounds for the CndG method applied to the minimum volume covering ellipsoid problem.
We also refer to \cite{HarJudNem12-1,Freund13} for some interesting convergence results for the CndG algorithm in terms of
the latter termination criterion.

\vgap

We first state a simple technical result.

\begin{lemma} \label{tech_result}
Let $\gamma_k\in (0,1]$, $k = 1, 2, \ldots$, be given. If the sequence $\{\Delta_k\}_{k \ge 0}$ satisfies
\beq \label{general_cond}
\Delta_k \le (1 - \gamma_k) \Delta_{k-1} + B_k, \ \ \ k = 1, 2, \ldots,
\eeq
then
\beq \label{general_result}
\Delta_k \le \Gamma_k (1-\gamma_1) \Delta_0 + \Gamma_k \sum_{i=1}^k \frac{B_i}{\Gamma_i},
\eeq
where 
\beq \label{def_Gamma0}
\Gamma_k := \left\{
\begin{array}{ll}
1,& k = 1,\\
(1-\gamma_k) \, \Gamma_{k-1},& k \ge 2.
\end{array}
\right.
\eeq
\end{lemma}

\begin{proof}
Dividing both sides of \eqnok{general_cond} by $\Gamma_k$, we obtain
\[
\frac{\Delta_1}{\Gamma_1} \le \frac{(1-\gamma_1) \Delta_{0}}{\Gamma_{1}} + \frac{B_1}{\Gamma_1}
\]
and
\[
\frac{\Delta_k}{\Gamma_k} \le \frac{\Delta_{k-1}}{\Gamma_{k-1}} + \frac{B_k}{\Gamma_k}, \ \ \ \forall k \ge 2.
\]
Summing up these inequalities, we obtain \eqnok{general_result}.
\end{proof}

\vgap

We are now ready to describe the main convergence properties of the CndG method.

\begin{theorem} \label{Theorem_FW}
Let $\{x_k\}$ be the sequence generated by the classic CndG method applied to problem \eqnok{cp}
with the stepsize policy in \eqnok{FW_step1} or \eqnok{FW_step2}. 
If $f(\cdot)$ satisfies \eqnok{smooth}, then for any $k = 1, 2, \ldots$,
\beq \label{Convergence_FW}
f(y_k) - f^* \le \frac{2L}{k(k+1)} \sum_{i=1}^k \|x_{i} - y_{i-1}\|^2 \le \frac{2 L D_X^2}{ k+1}.
\eeq
\end{theorem}

\begin{proof}
Let $\Gamma_k$ be defined in \eqnok{def_Gamma0} with
\beq \label{def_Gamma}
\gamma_k := \frac{2}{k+1}.
\eeq
It is easy to check that 
\beq \label{def_Gamma1}
\Gamma_k = \frac{2}{k (k+1)} \ \ \ \mbox{and} \ \ \ \frac{\gamma_k^2}{\Gamma_k} \le 2, \ \ \ k = 1, 2, \ldots.
\eeq
Denoting $\tilde y_k = (1-\gamma_k) y_{k-1} + \gamma_k x_k$, we conclude from
from \eqnok{FW_step1} (or \eqnok{FW_step2}) and the definition of $y_k$ in 
Algorithm~\ref{algFW} that $f(y_k) \le f(\tilde y_k)$.  
%
It also follows from the definition of $\tilde y_k$ that
$\tilde y_k - y_{k-1} = \gamma_k (x_k - y_{k-1})$. 
Letting $l_f(x;y)$ be defined in \eqnok{def_lf} and using these two observations, 
\eqnok{smoothness}, the definition of $x_k$ and
the convexity of $f(\cdot)$, we have 
\beqa
f(y_k) &\le& f(\tilde y_k) \le l_f(y_{k-1}; \tilde y_k) + \frac{L}{2} \|y_k - y_{k-1}\|^2 \nn\\
&=&  (1-\gamma_k) f(y_{k-1}) + \gamma_k l_f(y_{k-1}; x_k)
+ \frac{L }{2} \gamma_k^2\|x_k - y_{k-1}\|^2 \nn\\
&\le& (1-\gamma_k) f(y_{k-1}) + \gamma_k l_f(y_{k-1}; x) + \frac{L}{2} \gamma_k^2 \|x_k - y_{k-1}\|^2,\nn\\
&\le& (1-\gamma_k) f(y_{k-1}) + \gamma_k f(x) + \frac{L }{2} \gamma_k^2\|x_k - y_{k-1}\|^2,
\ \forall x \in X. \label{key_rel_FW0}
\eeqa
Subtracting $f(x)$ from both sides of the above inequality, we obtain
\beq \label{key_rel_FW}
f(y_k) - f(x) \le (1- \gamma_k) [f(y_{k-1}) - f(x)] + \frac{L}{2}  \gamma_k^2  \|x_k - y_{k-1}\|^2,
\eeq
which, in view of Lemma~\ref{tech_result}, then implies that
\beqa
f(y_k) - f(x) &\le& \Gamma_{k} (1-\gamma_1)  [f(y_{0}) - f(x)] + \frac{\Gamma_k L}{2} \sum_{i=1}^k \frac{\gamma_i^2}{\Gamma_i} \|x_i - y_{i-1}\|^2 \nn\\
&\le& \frac{2 L}{k (k+1)} \sum_{i=1}^k \|x_i - y_{i-1}\|^2 , \ \  k = 1, 2, \ldots,\label{key_rel_FW1}
\eeqa
where the last inequality follows from the fact that $\gamma_1 = 1$ and \eqnok{def_Gamma1}. 
\end{proof}

\vgap

We now add a few remarks about the results obtained in Theorem~\ref{Theorem_FW}.
Firstly, note that by \eqnok{Convergence_FW} and the definition of $D_{X}$ in \eqnok{def_DX}, 
we have, for any $k = 1, \ldots$,
\[
f(y_k) - f^* \le \frac{2L}{k+1} D_{X}^2.
\]
Hence, the number of iterations required by the classic CndG method to 
find an $\epsilon$-solution of problem~\eqnok{cp} is bounded by 
\beq \label{CG_bnd} 
{\cal O}(1) \frac{L D_{X}^2}{\epsilon}.
\eeq
Comparing the above bound with \eqnok{lowerbound}, we conclude that
the classic CndG algorithm is an optimal LCP method
for solving ${\cal F}_{L,\|\cdot\|}^{1,1}(X)$
if $n$ is sufficiently large.

Secondly, although the CndG method does not require the selection of
the norm $\|\cdot\|$, the iteration complexity of this algorithm, as stated in
\eqnok{CG_bnd}, does depend on $\|\cdot\|$ as the two constants, i.e., $L \equiv L_{\|\cdot\|}$ 
and $D_X \equiv D_{X,\|\cdot\|}$, depend on $\|\cdot\|$ . However, since
the result in \eqnok{CG_bnd} holds for an arbitrary $\|\cdot\|$,
the iteration complexity of the classic CndG method to solve problem~\eqnok{cp} can
actually be bounded by
\beq \label{norm_inv}
{\cal O}(1) \inf_{\|\cdot\|} \left\{ \frac{L_{\|\cdot\|} D_{X,\|\cdot\|}^2}{\epsilon} \right\}.
\eeq
For example, if $X$ is a simplex, a widely-accepted strategy to accelerate
gradient type methods is to set $\|\cdot\| = \|\cdot\|_1$ and
$d(x) = \sum_{i=1}^n x_i \log x_i$ in \eqnok{FO_subproblem}, in order to obtain
(nearly) dimension-independent complexity results,
which only grow mildly with the increase of the dimension of the problem (see \cite{nemyud:83,Nest05-1,Lan10-3}). On the other hand, 
the classic CndG method does automatically adjust to the geometry of the feasible set $X$ in order to 
obtain such scalability to high-dimensional problems (see Lemma 7 in \cite{Jaggi13} for some related discussions). 

Thirdly, observe that
the rate of convergence in \eqnok{Convergence_FW} depends on $\|x_k - y_{k-1}\|$ which usually
does not vanish as $k$ increases. For example, suppose $\{y_k\} \to x^*$ (this is true if $x^*$ is
a unique optimal solution of \eqnok{cp}), the distance $\{\|x_k - y_{k-1}\|\}$
does not necessarily converge to zero unless $x^*$ is an extreme point of $X$. In these cases,
the summation $\sum_{i=1}^k \|x_i - y_{i-1}\|^2$ increases linearly with respect $k$.
We will discuss some techniques in Section~\ref{sec_acLCP} that might help to improve
this situation.  


\subsection{Optimal CndG methods for saddle point problems under an $\LMO$ oracle} \label{sec_opt_sad}
In this subsection, we show that the CndG method, after incorporating some proper modification,
can achieve the optimal complexity for solving the saddle point problems ${\cal F}^{0}_{\|A\|}(X, Y)$ under an $\LMO$ oracle.
 
Since the objective function $f$ given by \eqnok{saddleclass} 
is nonsmooth in general, we cannot directly apply the CndG method
to ${\cal F}^{0}_{\|A\|}(X, Y)$. 
However, as shown by~\cite{Nest05-1}, the function $f(\cdot)$
in \eqnok{saddleclass} can be closely approximated by a class of smooth 
convex functions. More specifically,
for a given strongly convex function $v:Y \to \bbr$ such that
\beq \label{strong_v}
v(y) \ge v(x) + \langle v'(x), y - x \rangle + \frac{\sigma_v}{2} \|y -x \|^2, \forall x, y \in Y,
\eeq
let us denote $c_v := \argmin_{y \in Y} v(y)$, $V(y) := v(y) - v(c_v) - \langle \nabla v(c_v), y - c_v \rangle$ and
\beq \label{def_cal_DX}
{\cal D}_{Y,V}^2 := \max_{y \in Y} V(y).
\eeq 
Note that $V(y)$ is often referred to as the Bregman distance (from $y$ to $c_v$) in the literature,
which was initially studied by \cite{Breg67} and later by many others
(see \cite{AuTe06-1,BBC03-1,Kiw97-1,NJLS09-1} and references therein).
Then the function $f(\cdot)$ in \eqnok{saddleclass} can be closely approximated by
\beq \label{sm-approx}
f_\eta(x) := \max\limits_{y}\left\{\langle A x, y \rangle - 
	\hat{f}(y)-\eta \, [V(y) - {\cal D}_{Y,V}^2]: \ y \in Y \right\}.
\eeq
Indeed, by definition we have $0 \le V(y) \le {\cal D}_{Y,V}^2$ and hence, for any
$\eta \ge 0$,
\beq \label{closeness1}
f(x) \le f_\eta(x) \le f(x)+\eta \, {\cal D}_{Y,V}^2, \ \ \ \forall x \in X.
\eeq
Moreover, \cite{Nest05-1} shows that $f_{\eta}(\cdot)$ is differentiable and
its gradients are Lipschitz continuous with the Lipschitz constant given by 
\beq \label{new_ls}
{\cal L}_\eta := \frac{\|A\|^2}{\eta\sigma_v}.
\eeq

In view of this result, we modify the CndG method to solve ${\cal F}^{0}_{\|A\|}(X, Y)$
by replacing the gradient $f'(y_{k})$ in Algorithm~\ref{algFW} with
the gradient $f'_{\eta_k}(y_k)$ for some $\eta_k > 0$.
Observe that in the original Nesterov's smoothing scheme (\cite{Nest05-1}), 
we first need to define the smooth approximation function $f_\eta$ in \eqnok{sm-approx} by specifying 
in advance the smoothing parameter $\eta$ and then apply a smooth optimization
method to solve the approximation problem. The specification of $\eta$ usually requires explicit 
knowledge of $D_X$, ${\cal D}_{Y,V}^2$ and the target accuracy $\epsilon$ given a priori. 
However, by using a novel analysis, we show that one can use variable smoothing parameters $\eta_k$
and thus does not need to know the target accuracy $\epsilon$ in advance.
In addition, wrong estimation on $D_X$ and  ${\cal D}_{Y,V}^2$ only affects
the rate of convergence of the modified CndG method by a constant factor. Our analysis
relies on a slightly different construction of $f_\eta(\cdot)$ in \eqnok{sm-approx}
(i.e., the constant term $\eta {\cal D}_{Y,V}^2$ in \eqnok{sm-approx} does not appear in \cite{Nest05-1})
and the following simple observation.

\begin{lemma} \label{monotone}
Let $f_\eta(\cdot)$ be defined in \eqnok{sm-approx} and
$\eta_1 \ge \eta_2 \ge 0$ be given. Then, we have $
f_{\eta_1} (x) \ge f_{\eta_2}(x)$ for any $x \in X$.
\end{lemma}

\begin{proof}
The result directly follows from the definition of $f_\eta(\cdot)$ in \eqnok{sm-approx}
and the fact that $V(y) - {\cal D}_{Y,V}^2 \le 0$.
\end{proof} 

\vgap

We are now ready to describe the main convergence properties of this
modified CndG method to solve ${\cal F}^{0}_{\|A\|}(X, Y)$.

\begin{theorem} \label{the_saddle}
Let $\{x_k\}$ and $\{y_k\}$ be the two sequences generated by 
the CndG method with $f'(y_{k})$ replaced by $f_{\eta_{k}}(y_{k})$,
where $f_\eta(\cdot)$ is defined in \eqnok{saddleclass}.
If the stepsizes $\alpha_k$, $k = 1, 2, \ldots$, are
set to \eqnok{FW_step1} or \eqnok{FW_step2}, and $\{\eta_k\}$ satisfies 
\beq \label{cond_eta_k}
\eta_1 \ge \eta_2 \ge \ldots,
\eeq 
then we have, for any $k = 1, 2, \ldots$,
\beq \label{smoothing_CG}
f(y_k) - f^* \le \frac{2}{k (k+1)}\left[
\sum_{i=1}^k \left(i \eta_i {\cal D}_{Y,V}^2 + 
\frac{\|A\|^2}{\sigma_v \eta_i}  \|x_i - y_{i-1}\|^2\right) \right].
\eeq
In particular, if  
\beq \label{eta_selection}
\eta_k = \frac{\|A\| D_{X}}{ {\cal D}_{Y,V} \sqrt{\sigma_v k}}, 
\eeq
then we have, for any $k = 1, 2, \ldots$,
\beq \label{smoothing_CG_opt}
f(y_k) - f^* \le \frac{2 \sqrt{2} \|A\| D_{X} \, {\cal D}_{Y,V}}{\sqrt{\sigma_v k}},
\eeq
where $D_{X}$ and ${\cal D}_{Y,V}$ are defined in \eqnok{def_DX} and \eqnok{def_cal_DX}, respectively.
\end{theorem}

\begin{proof}
Let $\Gamma_k$ and $\gamma_k$ be defined in \eqnok{def_Gamma0} and \eqnok{def_Gamma},
respectively. Similarly to \eqnok{key_rel_FW},
we have, for any $x \in X$,
\beqas
f_{\eta_k}(y_k) &\le&  (1-\gamma_k) [f_{\eta_k}(y_{k-1})] + \gamma_k f_{\eta_k}(x) + \frac{L_{\eta_k}}{2} \gamma_k^2 \|x_k - y_{k-1}\|^2 \\
&\le& (1-\gamma_k) [f_{\eta_{k-1}}(y_{k-1})]+ \gamma_k f_{\eta_k}(x) + \frac{L_{\eta_k}}{2} \gamma_k^2 \|x_k - y_{k-1}\|^2 \\
&\le& (1-\gamma_k) [f_{\eta_{k-1}}(y_{k-1})] + \gamma_k [f(x) + \eta_k {\cal D}_{Y,V}^2] + \frac{L_{\eta_k}}{2} \gamma_k^2 \|x_k - y_{k-1}\|^2,
\eeqas
where the second inequality follows from \eqnok{cond_eta_k} and Lemma~\ref{monotone}, and
the third inequality follows from~\eqnok{closeness1}. 
Now subtracting $f(x)$ from both sides of the above inequality, we obtain, $\forall x \in X$,
\beqas
f_{\eta_k}(y_k) - f(x) &\le& (1-\gamma_k) [f_{\eta_{k-1}}(y_{k-1}) - f(x) ] + \gamma_k \eta_k {\cal D}_{Y,V}^2
+ \frac{L_{\eta_k}}{2} \gamma_k^2 \|x_k - y_{k-1}\|^2 \\
&\le& (1-\gamma_k) [f_{\eta_{k-1}}(y_{k-1}) - f(x) ] + \gamma_k \eta_k {\cal D}_{Y,V}^2
+ \frac{\|A\|^2 \gamma_k^2}{2 \sigma_v \eta_k}  \|x_k - y_{k-1}\|^2,
\eeqas
which, in view of Lemma~\ref{tech_result}, \eqnok{def_Gamma} and \eqnok{def_Gamma1}, then implies that,
$\forall x \in X$,
\[
f_{\eta_k}(y_k) - f(x) \le \frac{2}{k (k+1)}\left[
\sum_{i=1}^k \left(i \eta_i {\cal D}_{Y,V}^2 + 
\frac{\|A\|^2}{\sigma_v \eta_i}  \|x_i - y_{i-1}\|^2\right) \right], \ \ \forall k \ge 1.
\]
Our result in \eqnok{smoothing_CG} then immediately follows from \eqnok{closeness1} and the above inequality.
Now it is easy to see that the selection 
of $\eta_k$ in \eqnok{eta_selection} satisfies \eqnok{cond_eta_k}. By \eqnok{smoothing_CG} and
\eqnok{eta_selection}, we have
\beqas
f(y_k) - f^* &\le& \frac{2}{k (k+1)}\left[
\sum_{i=1}^k \left(i \eta_i {\cal D}_{Y,V}^2 + 
\frac{\|A\|^2}{\sigma_v \eta_i}  D_{X}^2\right) \right]\\
&=& \frac{4 \|A\| D_{X} {\cal D}_{v,Y}}{k (k+1) \sqrt{\sigma_v }}
\sum_{i = 1}^k \sqrt{i} \le \frac{8 \sqrt{2} \|A\| D_{X} {\cal D}_{Y,V}}{3 \sqrt{\sigma_v k}},
\eeqas
where the last inequality follows from the fact that
\beq \label{simple}
\sum_{i = 1}^k \sqrt{i} \le \int_{0}^{k+1} t dt \le \frac{2}{3} (k+1)^\frac{3}{2}
\le \frac{2 \sqrt{2}}{3} (k+1) \sqrt{k}.
\eeq
\end{proof}

A few remarks about the results obtained in Theorem~\ref{the_saddle} are in order.
First, observe that the specification of $\eta_k$ in \eqnok{eta_selection}
requires the estimation of a few problem parameters, including $\|A\|$, $D_{X}$,
${\cal D}_{Y,V}$ and $\sigma_v$. However, wrong estimation on these parameters
will only result in the increase on the rate of convergence of the modified CndG method
by a constant factor. For example, if $\eta_k = 1/\sqrt{k}$ for any $
k \ge 1$, then \eqnok{smoothing_CG} reduces to
\[
f(y_k) - f^* \le \frac{8\sqrt{2}}{3 \sqrt{k}} 
\left( {\cal D}_{Y,V}^2 + \frac{\|A\|^2 D_{X}^2}{\sigma_v}\right).
\]
It is worth noting that similar adaptive smoothing schemes can also
be used when one applies Nesterov's accelerated gradient method to solve ${\cal F}^{0}_{\|A\|}(X, Y)$.
Second, suppose that the norm $\|\cdot\|$ in the dual space associated with $Y$
is an inner product norm and $v(y) = \|y\|^2/2$. In this case, by the definitions
of $D_{Y}$ and ${\cal D}_{Y,V}$ in \eqnok{def_DY} and \eqnok{def_cal_DX}, 
we have ${\cal D}_{Y,V} \le D_{Y}$. Using this observation and \eqnok{smoothing_CG_opt},
we conclude that the number of iterations required by the modified CndG method
to solve ${\cal F}^{0}_{\|A\|}(X, Y)$ can be bounded by
\[
{\cal O}(1) \left( \frac{\|A\| D_{X} D_{Y}}{\epsilon} \right)^2,
\]
which, in view of Theorem~\ref{the_lb_nonsmooth}, is optimal when $n$ is sufficiently large.


\subsection{Nearly optimal CndG  methods for general nonsmooth problems under an $\LMO$ oracle} \label{sec_opt_nonsm}
In this subsection, we present a randomized CndG method and demonstrate that it can achieve a nearly optimal rate of convergence
for solving general nonsmooth CP problems ${\cal F}^{0}_{M,\|\cdot\|}(X)$
under an $\LMO$ oracle. To the best of our knowledge, no such CndG methods have not been 
presented before for solving general nonsmooth CP problems in the literature.

The basic idea is to approximate the general nonsmooth CP problems 
${\cal F}^{0}_{M,\|\cdot\|}(X)$ by using the convolution-based
smoothing. The intuition underlying such a approach is that convolving two functions yields 
a new function that is at least as smooth as the smoother one of the original two functions.
In particular, let $\mu$ denote the density of a random variable with respect to 
Lebesgue measure and consider the function $f_\mu$ given by
\[
f_\mu(x) := (f * \mu)(x) = \int_{\bbr^n} f(y) \mu(x-y) d(y) = \bbe_\mu[x+Z], 
\]
where $Z$ is a random variable with density $\mu$. Since $\mu$ is 
a density with respect to Lebesgue measure, $f_\mu$ is differentiable~(\cite{Bert73-1}).
The above convolution-based smoothing technique has been extensively studied
in stochastic optimization, e.g.,~\cite{Bert73-1,DuBaMaWa11,KatKul72-1,Nest11-1,Rubin81}.
For the sake of simplicity, we assume throughout this subsection that
$\|\cdot\| = \|\cdot\|_2$ and $Z$ is uniformly distributed over a certain Euclidean ball.
The following result is known in the literature (see, e.g., \cite{DuBaMaWa11}).

\vgap

\begin{lemma} \label{rand_lemma}
Let $\xi$ be uniformly distributed over the $l_2$-ball ${\cal B}_2(0,1) := \{x \in \bbr^n: \|x\|_2 \le 1\}$
and $u >0$ is given. Suppose that \eqnok{nonsmooth} holds for any $x,y \in X + u \, {\cal B}_2(0,1)$. 
Then, the following statements hold for the function $f_u(\cdot)$
given by
\beq \label{con_smooth}
f_u(x) := \bbe[f(x + u \xi)].
\eeq
\begin{itemize}
\item [a)] $f(x) \le f_u(x) \le f(x) + M u $;
\item [b)] $f_u(x)$ has $M \sqrt{n}/u$-Lipschitz continuous gradient with respect to $\|\cdot\|_2$;
\item [c)] $\bbe[f'(x + u \xi)] = f'_u(x)$ and
$\bbe[\|f'(x + u \xi) - f'_u(x)\|^2] \le M^2$;
\item [d)] If $u_1 \ge u_2 \ge 0$, then $f_{u_1} (x) \ge f_{u_2}(x)$ for any $x \in X$.
\end{itemize}
\end{lemma}

\vgap

In view of the above result, we can apply the CndG method directly to $\min_{x \in X} f_u(x)$
for a properly chosen $\mu$ in order to solve the original problem \eqnok{cp}. The only problem
is that we cannot compute the gradient of $f_u(\cdot)$ exactly. To address this issue,
we will generate an i.i.d. random sample $(\xi_1, \ldots, \xi_T)$ for some $T > 0$ and
approximate the gradient $f'_\mu(x)$ by
$
\tilde f'_{u}(x) := \frac{1}{T}\sum_{t=1}^{T} f'(x, u \xi_t).
$
After incorporating the aforementioned
randomized smoothing scheme, the CndG method applied to ${\cal F}^0_{M,\|\cdot\|}(X)$ exhibits the following convergence properties.

\begin{theorem} \label{the_randsmooth}
Let $\{x_k\}$ and $\{y_k\}$ be the two sequences generated by 
the classic CndG method with $f'(y_{k-1})$ replaced by
the average of the sampled gradients, i.e.,
\beq \label{random_sample}
\tilde f'_{u_k}(y_{k-1}) := \frac{1}{T_k}\sum_{t=1}^{T_k} f'(y_{k-1} + u_k \xi_t).
\eeq
where $f_u$ is defined in \eqnok{con_smooth} and $\{\xi_1, \ldots, \xi_{T_k}\}$ 
is an i.i.d. sample of $\xi$. If the stepsizes $\alpha_k$, $k = 1, 2, \ldots$, are
set to \eqnok{FW_step1} or \eqnok{FW_step2}, and
$\{u_k\}$ satisfies
\beq \label{cond_u_k}
u_1 \ge u_2 \ge \ldots,
\eeq
then we have
\beq \label{randsmoothing_CG}
\bbe[f(y_k) ] - f(x) \le \frac{2 M}{k (k+1)} 
\left[\sum_{i=1}^k \left( \frac{i}{\sqrt{T_i}} D_{X} + 
i u_i + \frac{\sqrt{n}}{u_i} D_{X}^2\right)\right],
\eeq
where $M$ is given by \eqnok{nonsmooth}.
In particular, if
\beq \label{u_selection}
T_k = k \ \ \ \mbox{and} \ \ \ u_k = \frac{n^\frac{1}{4} D_{X}}{\sqrt{k}},
\eeq
then
\beq \label{randsmoothing_CG_opt}
\bbe[f(y_k) ] - f(x) \le \frac{4 (1+ 2n^\frac{1}{4}) M D_{X}}{3 \sqrt{k}},
k = 1, 2, \ldots.
\eeq
\end{theorem}

\begin{proof}
Let $\gamma_k$ be defined in \eqnok{def_Gamma}, similarly to \eqnok{key_rel_FW0},
we have
\beqa 
f_{u_k}(y_k) &\le& (1 - \gamma_k) f_{u_k}(y_{k-1}) + \gamma_k l_{f_{u_k}} (x_k; y_{k-1}) 
+ \frac{M \sqrt{n}}{2u_k} \gamma_k^2 \|x_k - y_{k-1}\|^2 \nn\\
&\le& (1 - \gamma_k) f_{u_{k-1}}(y_{k-1}) + \gamma_k l_{f_{u_k}} (x_k; y_{k-1}) 
+ \frac{M \sqrt{n}}{2u_k} \gamma_k^2 \|x_k - y_{k-1}\|^2, \label{rand_basic}
\eeqa
where the last inequality follows from the fact that $f_{u_{k-1}}(y_{k-1}) \ge f_{u_k}(y_{k-1})$
due to Lemma~\ref{rand_lemma}.d).
Let us denote $\delta_k := f'_{u_k}(y_{k-1}) - \tilde f'_{u_k}(y_{k-1}) $. Noting that 
by definition of $x_k$ and the convexity of $f_{u_k}(\cdot)$, 
\beqas
l_{f_{u_k}} (x_k; y_{k-1}) &=& f_{u_k}(y_{k-1}) + \langle f_{u_k}'(y_{k-1}, x_k - y_{k-1} \rangle\\
&=& f_{u_k}(y_{k-1}) + \langle \tilde f_{u_k}'(y_{k-1}), x_k - y_{k-1} \rangle + \langle \delta_k, x_k - y_{k-1} \rangle\\
&\le& f_{u_k}(y_{k-1}) + \langle \tilde f_{u_k}'(y_{k-1}), x - y_{k-1} \rangle + \langle \delta_k, x_k - y_{k-1} \rangle\\
&=& f_{u_k}(y_{k-1}) + \langle f_{u_k}'(y_{k-1}), x - y_{k-1} \rangle +  \langle \delta_k, x_k - x \rangle\\ 
&\le& f_{u_k}(x) + \|\delta_k\| D_{X}
\le f(x) + \|\delta_k\| D_{X} + M u_k, \ \ \forall x \in X,
\eeqas
where the last inequality follows from Lemma~\ref{rand_lemma}.a),
we conclude from \eqnok{rand_basic} that, $\forall x \in X$, 
\[
f_{u_k}(y_k) \le (1 - \gamma_k) f_{u_{k-1}}(y_{k-1}) + \gamma_k \left[  f(x) + \|\delta_k\| D_{X} + M u_k\right]
+ \frac{M \sqrt{n}}{2u_k} \gamma_k^2 \|x_k - y_{k-1}\|^2,
\]
which implies that
\[
f_{u_k}(y_k) - f(x) \le (1 - \gamma_k) [f_{u_{k-1}}(y_{k-1}) - f(x)] + \gamma_k \left[\|\delta_k\| D_{X} + M u_k \right]
+ \frac{M \sqrt{n}}{2u_k} \gamma_k^2 D_{X}^2,
\] 
Noting that by Jensen's inequality and
Lemma~\ref{rand_lemma}.c), 
\beq \label{bound_delta}
\left\{\bbe[\|\delta_k\|]\right\}^2 \le \bbe[\|\delta_k\|^2] = \frac{1}{T_k^2} \sum_{t=1}^{T_k}\bbe[\|f'(y_{k-1}+u_k \xi_k) - f'_{u_k}(y_{k-1})\|^2]
\le \frac{M^2}{T_k},
\eeq
we conclude from the previous inequality that
\[
\bbe[f_{u_k}(y_k) - f(x)] \le (1 - \gamma_k) \bbe[f_{u_{k-1}}(y_{k-1}) - f(x)] + \frac{\gamma_k}{\sqrt{T_k}} M D_{X}
+ M \gamma_k u_k
+ \frac{M \sqrt{n}}{2u_k} \gamma_k^2 D_{X}^2,
\] 
which, in view of Lemma~\ref{tech_result}, \eqnok{def_Gamma} and \eqnok{def_Gamma1}, then implies that,
$\forall x \in X$,
\[
\bbe[f_{u_k}(y_k) - f(x)] \le \frac{2}{k (k+1)} \left[\sum_{i=1}^k \left( \frac{i}{\sqrt{T_i}} M D_{X} + 
 M i u_i + \frac{M \sqrt{n}}{u_i} D_{X}^2\right)\right]
\]
The result in \eqnok{randsmoothing_CG} follows directly from Lemma~\ref{rand_lemma}.a) and the above inequality.
Using \eqnok{simple}, \eqnok{randsmoothing_CG} and \eqnok{u_selection}, we can easily verify that
the bound in \eqnok{randsmoothing_CG_opt} holds.
\end{proof}

\vgap

We now add a few remarks about the results obtained in Theorem~\ref{the_randsmooth}.
Firstly, note that in order to obtain the result in \eqnok{randsmoothing_CG_opt},
we need to set $T_k = k$. This implies that at the $k$-th iteration of the randomized CndG method
in Theorem~\ref{the_randsmooth}, we need to take an i.i.d. 
sample $\{\xi_1, \ldots, \xi_k\}$ of $\xi$ and compute the corresponding 
gradients $\{f'(y_{k-1}, \xi_1), \ldots, f'(y_{k-1}, \xi_k) \}$. Also note that
from the proof of the above result, we can recycle the generated samples 
$\{\xi_1, \ldots, \xi_k\}$ for usage in subsequent iterations.  

Secondly, since ${\cal F}^0_{\|A\|}(X,Y) \subset {\cal F}^0_{M,\|\cdot\|}(X)$,
we can apply the randomized CndG method to solve the saddle point problems ${\cal F}^0_{\|A\|}(X,Y)$.
In comparison with the smoothing CndG method in Subsection~\ref{sec_opt_sad}, we do not need to solve
the subproblems given in the form of \eqnok{sm-approx}, but to solve the subproblems
\[
\max_y \left\{ \langle A (x + \xi_i), y \rangle - \hat f(y): y \in Y \right\},
\]
in order to compute $f'(y_{k-1}, \xi_i)$, $i = 1, \ldots,k$, at the $k$-th iteration. In particular, if $\hat f(y) = 0$,
then we only need to solve linear optimization subproblems over the set $Y$. 
To the best of our knowledge, this is the first time that optimization algorithms of
this type has been proposed in the literature (see discussions in Section 1 of \cite{Nest08-1}).

Thirdly, in view of \eqnok{randsmoothing_CG_opt}, the number of iterations (calls to the $\LMO$ oracle)
required by the randomized CndG method to find a solution $\bar x$ such that $\bbe[f(\bar x) - f^*] \le \epsilon$
can be bounded by 
\beq \label{bnd_nonsmooth_LO}
N_\epsilon := {\cal O}(1) \frac{\sqrt{n} M^2 D_{X}^2}{\epsilon^2},
\eeq
and that the total number of subgradient evaluations can be bounded by 
\[
\sum_{k=1}^{N_\epsilon} T_k = \sum_{k=1}^{N_\epsilon} k = {\cal O}(1) N_\epsilon^2.
\]
According to the lower complexity bound in \eqnok{lb_nonsmooth}, we conclude that the above complexity bound in 
\eqnok{bnd_nonsmooth_LO} is
nearly optimal for the following reasons: i) the above result is in the the same order of magnitude
as \eqnok{lb_nonsmooth} with an additional factor of $\sqrt{n}$;
and ii) the termination criterion is in terms of expectation. Note that while it is possible to
show that the relation \eqnok{randsmoothing_CG_opt} holds with overwhelming probability by
developing certain large deviation results associated with \eqnok{randsmoothing_CG_opt}, such
a result has been skipped in this paper for the sake of simplicity, see, e.g.,
\cite{GhaLan12} for some similar developments.

\subsection{CndG methods for strongly convex problems under an enhanced $\LMO$ oracle} \label{sec_opt_strong}
In this subsection, we assume that the objective function $f(\cdot)$ in \eqnok{cp} is smooth and 
strongly convex, i.e., in addition to \eqnok{smooth}, it also satisfies  
\beq \label{strongly_convex}
f(y) - f(x) - \langle f'(x), y - x \rangle \ge \frac{\mu}{2} \|y - x\|^2, \ \forall x, y \in X.
\eeq
These problems have been extensively studied in the literature. For example, it has been shown in \cite{Nest83-1,Nest04} that
the optimal complexity for the general first-order methods to solve this class of problems is given by
by 
\[
{\cal O} (1) \sqrt{\frac{L}{\mu} } \max\left(\log \frac{\mu D_{X}}{\epsilon}, 1\right).
\] 
On the other hand, as noted in Subsection~\ref{sec_smooth}, the number of calls to the $\LMO$ oracle for
the LCP methods to solve these problems cannot be smaller than ${\cal O}(L D_{X}^2/\epsilon)$.

Our goal in this subsection is to show that, under certain stronger assumptions on
the $\LMO$ oracle, we can somehow ``improve'' the complexity of the CndG method for solving 
these strongly convex problems. More specifically, we assume throughout this subsection 
that we have access to an enhanced $\LMO$ oracle, which can solve
optimization problems given in the form of
\beq \label{new_LP}
\min \left\{ \langle p, x \rangle: x \in X, \|x - x_0\| \le R \right\}
\eeq
for any given $x_0 \in X$.
For example, we can assume that the norm $\|\cdot\|$ is chosen such that
problem~\eqnok{new_LP} is relatively easy to solve. In particular, if $X$ is a polytope, we can set
$\|\cdot\| = \|\cdot\|_\infty$ or $\|\cdot\| = \|\cdot\|_1$ and then the complexity to solve
\eqnok{new_LP} will be comparable to the one to solve \eqnok{CG_subproblem}.
Note however, that such a selection of $\|\cdot\|$ will possibly increase the value of
the condition number given by $L /\mu$. Motivated by \cite{GhaLan10-1b}, we present
a shrinking CndG method under the above assumption on the enhanced $\LMO$ oracle.
It should be noted that the linear rate of convergence under stronger assumptions
of the problem and/or the $\LMO$ oracle is not completely new in the literature.
More specifically,
we notice that Garber and Hanzan~\cite{GarberHazan13} have
made some interesting development for the CndG methods applied to strongly convex problems,
although the algorithm and analysis given here seem to be different from those in \cite{GarberHazan13}.
In addition, the linear convergence of the CndG method has been shown in \cite{PshDan78} for the case when
the feasible set $X$ is round (strongly convex as a set), and similar result has been generalized in \cite{GuMar86} for
the case when the optimal solution resides in the interior of the feasible set.

\begin{algorithm} [H]
	\caption{The Shrinking Conditional Gradient (CndG) Method}
	\label{algFWS}
	\begin{algorithmic}
\STATE Let $p_0\in X$ be given. Set $R_0 = D_{X}$. 
\FOR {$t=1, \ldots$}
\STATE Set $y_0 = p_{t-1}$.
\FOR {$k=1, \ldots, 8L/\mu$ }
\STATE Call the enhanced $\LMO$ oracle to compute 
$x_{k} \in \Argmin_{x \in X_{t-1}} \langle f'(y_{k-1}), x \rangle$,\\
where $X_{t-1}:= \left\{ x \in X: \|x - p_{t-1}\| \le R_{t-1}\right\}$. 
\STATE Set $y_{k}  = (1 - \alpha_{k}) y_{k-1} + \alpha_{k} x_{k}$ for some $\alpha_k \in [0,1]$.
\ENDFOR
\STATE Set $p_t = y_k$ and $R_t = R_{t-1}/\sqrt{2}$;
\ENDFOR
	\end{algorithmic}
\end{algorithm}

Note that an outer (resp., inner) iteration of the above shrinking CndG method
occurs whenever $t$ (resp., $k$) increases by $1$. Observe also that the feasible 
set $X_t$ will be reduced at every outer iteration $t$. 
The following result summarizes the convergence properties for this algorithm.

\begin{theorem}
Suppose that conditions \eqnok{smooth} and \eqnok{strongly_convex} hold.
If the stepsizes $\{\alpha_k\}$ in the shrinking CndG method  are set to \eqnok{FW_step1}
or \eqnok{FW_step2}, then the number of calls to the enhanced $\LMO$ oracle performed 
by this algorithm to find an $\epsilon$-solution
of problem \eqnok{cp} can be bounded by
\beq \label{ub_strong}
\frac{8L}{\mu} \left \lceil \max\left( \log \frac{\mu R_0 }{ \epsilon}, 1\right) \right\rceil.
\eeq
\end{theorem}

\begin{proof}
Denote $K \equiv 8 L/ \mu$. We first claim that $x^* \in X_t$ for any $t \ge 0$. This relation is obviously true for $t = 0$
since $\|y_0 - x^*\| \le R_0 = D_{X}$.
Now suppose that $x^* \in X_{t-1}$ for some $t \ge 1$. Under this assumption, 
relation \eqnok{key_rel_FW1} holds with $x = x^*$ for inner iterations $k = 1, \ldots, K$ 
performed at the $t$-th outer iteration. Hence, we have
\beq \label{onephase}
f(y_k) - f(x^*) \le \frac{2L}{k (k+1)} \sum_{i=1}^k \|x_i - y_{i-1}\|^2
\le \frac{2L}{k+1} R_{t-1}^2, \ \ k = 1, \ldots, K.
\eeq
Letting $k = K$ in the above relation, and using the facts that
$p_t = y_K$ and $f(y_K) - f^* \ge \mu \|y_K - x^*\|^2 /2$, we conclude that
\beq \label{inclusion}
\|p_t - x^*\|^2 \le \frac{2}{\mu}[f(p_t) - f^*] =\frac{2}{\mu}[f(y_K) - f^*] 
\le \frac{4 L}{\mu (K+1)} R_{t-1}^2 \le \frac{1}{2} R_{t-1}^2 = R_t^2,
\eeq
which implies that $x^* \in X_t$. We now provide a bound on the total number of
calls to the $\LMO$ oracle (i.e., the total number of inner iterations)
performed by the shrinking CndG method. It follows from \eqnok{inclusion}
and the definition of $R_t$ that
\[
f(p_t) - f^* \le \frac{\mu}{2} R_t^2 = \frac{\mu}{2} \frac{R_0}{2^{t-1}}, \ \ \
t = 1, 2, \ldots.
\]
Hence the total number of outer iterations performed by the shrinking CndG method
for finding an $\epsilon$-solution of \eqnok{cp} is bounded by
$\lceil \max( \log \mu R_0 / \epsilon, 1) \rceil$. This observation, in view of the fact
that $K$ inner iterations are performed at each outer iteration $t$,
then implies that the total number of inner iterations is bounded by \eqnok{ub_strong}.
\end{proof}

\setcounter{equation}{0}
\section{New variants of for LCP methods} \label{sec_acLCP}
Our goal in this section is to present a few new LCP methods for CP, obtained by
replacing the projection (prox-mapping) subproblems with linear optimization
subproblems in Nesterov's accelerated gradient method. 
Throughout this section, 
we focus on smooth CP problems ${\cal F}^{1,1}_{L,\|\cdot\|}(X)$. However, the developed
algorithms can be easily modified to solve saddle point problems,
general nonsmooth CP problems and strongly convex problems, by using similar ideas to
those described in Section~\ref{sec_ub}.


\subsection{Primal averaging CndG method} \label{sec_primal}
In this subsection, we present a new LCP method, obtained by incorporating a 
primal averaging step into the CndG method. This algorithm is formally
described as follows.

\begin{algorithm} [H]
	\caption{The Primal Averaging Conditional Gradient (PA-CndG) Method}
	\label{algPACG}
	\begin{algorithmic}
\STATE Let $x_0 \in X$ be given. Set $y_0 = x_0$.

\FOR {$k=1, \ldots$ }
\STATE Set $z_{k-1} = \frac{k-1}{k+1} y_{k-1} + \frac{2}{k+1} x_{k-1}$ and $p_k = f'(z_{k-1})$.
\STATE Call the $\LMO$ oracle to compute $x_{k} \in \Argmin_{x\in X}\langle p_k, x \rangle$. 
\STATE Set $y_{k}  = (1 - \alpha_{k}) y_{k-1} + \alpha_{k} x_{k}$ for some $\alpha_k \in [0,1]$.
\ENDFOR
	\end{algorithmic}
\end{algorithm}

It can be easily seen that the PA-CndG method stated above 
is a special case of the LCP method in Algorithm~\ref{algGeneric}. 
It differs from the classic CndG method in the way
that the search direction $p_k$ is defined. In particular, 
while $p_k$ is set to $f'(x_{k-1})$ in the classic CndG algorithm, the search direction 
$p_k$ in PA-CndG is given by $f'(z_{k-1})$ for some $z_{k-1} \in \conv\{x_0, x_1, \ldots, x_{k-1}\}$.
In other words, we will need to ``average'' the primal sequence $\{x_k\}$
before calling the $\LMO$ oracle to update the iterates.
It is worth noting that the PA-CndG method can be viewed as a variant of
Nesterov's method in~\cite{Nest04,Lan10-3}, obtained by replacing
the projection (or prox-mapping) subproblem with a simpler linear optimization
subproblem.

\vgap

By properly choosing the stepsize parameter $\alpha_k$, 
we have the following convergence results for the PA-CndG method described above.

\begin{theorem} \label{Theorem_PACG}
Let $\{x_k\}$ and $\{y_k\}$ be the sequences generated by the PA-CndG method applied to
problem~\eqnok{cp} with the stepsize policy in \eqnok{FW_step1} or \eqnok{FW_step2}.
Then we have
\beq \label{Convergence_PACG}
f(y_k) - f^* \le \frac{2 L}{k (k+1)} \sum_{i=1}^k \|x_i - x_{i-1}\|^2 \le \frac{2 L D_X^2}{k+1}, \ \ k = 1, 2, \ldots,
\eeq
where $L$ is given by \eqnok{smoothness}.
\end{theorem}

\begin{proof}
Let $\gamma_k$ and $\Gamma_k$ be defined in \eqnok{def_Gamma0} and \eqnok{def_Gamma}, respectively.
Denote $\tilde y_k = (1-\gamma_k) y_{k-1} + \gamma_k x_k$. 
It can be easily seen from \eqnok{FW_step1} (or \eqnok{FW_step2}) and the
definition of $y_k$ in Algorithm~\ref{algPACG} that
$f(y_k) \le f(\tilde y_k)$. Also by definition, we have
$z_{k-1} = (1-\gamma_k) y_{k-1} + \gamma_k x_{k-1}$ and hence
\[
\tilde y_k - z_{k-1} = \gamma_k (x_k - x_{k-1}).
\]
Letting $l_f(\cdot, \cdot)$ be defined in \eqnok{def_lf}, and using the previous two observations, 
\eqnok{smoothness}, the definition of $x_k$ in Algorithm~\ref{algPACG},
and the convexity of $f(\cdot)$, we obtain
\beqa
f(y_k) &\le& f(\tilde y_k) \le 
l_f(z_{k-1}; \tilde y_k) + \frac{L}{2} \|\tilde y_k - z_{k-1}\|^2 \nn\\
&=& (1 - \gamma_k) l_f(z_{k-1}; y_{k-1}) +  \gamma_k l_f(z_{k-1}; x_k ) 
+ \frac{L}{2} \gamma_k^2 \|x_k - x_{k-1}\|^2 \nn \\
&\le& (1-\gamma_k) f(y_{k-1}) + \gamma_k l_f(z_{k-1}; x)
+ \frac{L}{2} \gamma_k^2 \|x_k - x_{k-1}\|^2 \nn \\
&\le& (1-\gamma_k) f(y_{k-1}) + \gamma_k f(x) + \frac{L}{2} \gamma_k^2 \|x_k - x_{k-1}\|^2.
\eeqa
Subtracting $f(x)$ from both sides of the above inequality, we have
\[
f(y_k) - f(x) \le (1-\gamma_k) [f(y_{k-1}) - f(x)] + \frac{L}{2} \gamma_k^2 \|x_k - x_{k-1}\|^2,
\]
which, in view of Lemma~\ref{tech_result}, \eqnok{def_Gamma1} and the fact that
$\gamma_1 = 1$, then implies that, $\forall x \in X$,
\beqas
f(y_k) - f(x) &\le& \Gamma_k (1-\gamma_1) [f(y_0) - f(x)] + \frac{\Gamma_k L}{2}
\sum_{i=1}^k \frac{\gamma_i^2}{\Gamma_i} \|x_i - x_{i-1}\|^2\\
&\le& \frac{2L}{k(k+1)} \sum_{i=1}^k \|x_i - x_{i-1}\|^2, \ \ k = 1, 2, \ldots.
\eeqas
\end{proof}

\vgap

We now add a few remarks about the results obtained in Theorem~\ref{Theorem_PACG}.
Firstly, similarly to \eqnok{CG_bnd}, we can easily see that the number of iterations
required by the PA-CndG method to find an $\epsilon$-solution of problem \eqnok{cp}
is bounded by ${\cal O}(1) L D_{X}^2/\epsilon$. Therefore, the PA-CndG method
is an optimal LCP method for solving ${\cal F}^{1,1}_{L,\|\cdot\|}(X)$ when $n$ is sufficiently
large. In addition, since the selection of $\|\cdot\|$ is arbitrary, the iteration 
complexity of this method can also be bounded by \eqnok{norm_inv}.

Secondly, while the rate of convergence for the CndG method (cf. \eqnok{Convergence_FW})
depends on $\|x_k - y_{k-1}\|$, the one for the PA-CndG method
depends on $\|x_k - x_{k-1}\|$, i.e., the distance between the output of the $\LMO$ oracle in
two consecutive iterations. 
Clearly, the distance $\|x_k - x_{k-1}\|$ will depend on the geometry of $X$ and the
difference between $p_k$ and $p_{k-1}$.
Let $\gamma_k$ be defined in \eqnok{def_Gamma} and suppose that
$\alpha_k$ is set to \eqnok{FW_step1} (i.e., $\alpha_k = \gamma_k$). Observe that 
by definitions of $z_k$ and $y_k$ in Algorithm~\ref{algPACG}, we have
\beqas
z_k - z_{k-1} &=& (y_k - y_{k-1}) + \gamma_{k+1} (x_k - y_k) - \gamma_k (x_{k-1} - y_{k-1})\\
&=& \alpha_k (x_k - y_{k-1}) + \gamma_{k+1} (x_k - y_k) - \gamma_k (x_{k-1} - y_{k-1})\\
&=& \gamma_k (x_k - y_{k-1}) + \gamma_{k+1} (x_k - y_k) - \gamma_k (x_{k-1} - y_{k-1}),
\eeqas 
which implies that $\|z_k - z_{k-1}\| \le 3 \gamma_k D_{X}$.
Using this observation, \eqnok{smooth} and the definition of $p_k$, we have
\beq \label{PA_closeness}
\|p_k - p_{k-1}\|_* = \|f'(z_{k-1}) - f'(z_{k-2})\|_* \le 3 \gamma_{k-1} L D_X. 
\eeq 
Hence, the difference between $p_k$ and $p_{k-1}$ vanishes as $k$ increases.
By exploiting this fact, we establish in 
Corollary~\ref{cor_PACG} certain necessary conditions about the $\LMO$ oracle, under which the
rate of convergence of the PA-CndG algorithm can be improved. It should be noted, however, that this result
is more of theoretical interest only, since these assumptions on the $\LMO$ oracle are quite strong
and hard to be satisfied over a global scope.
\begin{corollary} \label{cor_PACG}
Let $\{y_k\}$ be the sequence generated by the PA-CndG method applied to
problem~\eqnok{cp} with the stepsize policy in \eqnok{FW_step1}.
Suppose that the $\LMO$ oracle satisfies
\beq \label{Lipschitz_LO}
\|x_k - x_{k-1}\| \le Q \|p_k - p_{k-1}\|_*^\rho, \ \ k \ge 2,
\eeq
for some $\rho \in (0,1]$ and $Q > 0$.
Then we have, for any $k \ge 1$,
\beq \label{Convergence_PACG1}
f(y_k) - f^* \le
 {\cal O}(1) \left\{
\begin{array}{ll}
Q^2 L^{2\rho+1} D_{X}^{2 \rho} / \left[(1-2 \rho) \, k^{2 \rho +1}\right], & \rho \in (0, 0.5),\\
Q^2 L^{2} D_{X} \log (k+1)/ k^{2}, & \rho = 0.5,\\
Q^2 L^{2\rho+1} D_{X}^{2 \rho}/ \left[(2 \rho - 1) \, k^2\right], &\rho \in (0.5,1].
\end{array}
\right.
\eeq
\end{corollary}

\begin{proof}
Let $\gamma_k$ be defined in \eqnok{def_Gamma}.
By \eqnok{PA_closeness} and \eqnok{Lipschitz_LO}, we have
\[
\|x_k-x_{k-1}\| \le Q \|p_k - p_{k-1}\|_*^\rho 
\le Q (3 \gamma_k L D_{X})^\rho
\]
for any $k \ge 2$. The result follows by plugging the above bound into \eqnok{Convergence_PACG}
and noting that 
\[
\sum_{i=1}^k (i+1)^{-2\rho}  \le 
\left\{
\begin{array}{ll}
\frac{(k+1)^{-2 \rho +1}}{1 -2 \rho}, & \rho \in (0,0.5) ,\\
\log (k+1), & \rho = 0.5,\\
\frac{1}{2\rho - 1}, & \rho \in (0.5,1].
\end{array}
\right.
\]
\end{proof}

\vgap

The bound obtained in \eqnok{Convergence_PACG1} provides some interesting insights on
the relation between first-order LCP methods and the general optimal first-order methods for CP.
More specifically, if the $\LMO$ oracle satisfies the H\"{o}lder's continuity condition \eqnok{Lipschitz_LO} 
for some $\rho \in (0.5,1]$, then we can obtain an ${\cal O}(1/k^2)$ rate of convergence for the 
PA-CndG method for solving ${\cal F}^{1,1}_{L, \|\cdot\|}(X)$.

\subsection{Primal-dual averaging CndG methods} \label{sec_dual}
Our goal in this subsection is to present another new LCP method, namely
the primal-dual averaging CndG method, obtained by introducing a different acceleration scheme
into the CndG method. This algorithm is formally described as follows.

\begin{algorithm} [H]
	\caption{The Primal-Dual Averaging Conditional Gradient (PDA-CndG) Method}
	\label{algPDACG}
	\begin{algorithmic}
\STATE Let $x_0 \in X$ be given and set $y_0 = x_0$.

\FOR {$k=1, \ldots$ }
\STATE Set $z_{k-1} = \frac{k-1}{k+1} y_{k-1} + \frac{2}{k+1} x_{k-1}$.
\STATE Set $p_k = \Theta_k^{-1} \sum_{i=1}^{k} [\theta_i f'(z_{i-1})]$, where $\theta_i \ge 0$ are given
and $\Theta_k = \sum_{i=1}^k \theta_i$.
\STATE Call the $\LMO$ oracle to compute $x_{k} \in \Argmin_{x\in X}\langle p_k, x \rangle$.
\STATE Set $y_{k}  = (1 - \alpha_{k}) y_{k-1} + \alpha_{k} x_{k}$ for some $\alpha_k \in [0,1]$.
\ENDFOR
	\end{algorithmic}
\end{algorithm}

\vgap

Clearly, the above PDA-CndG method is also a special LCP algorithm.
While the input vector $p_k$ to the $\LMO$ oracle is set to $f'(z_{k-1})$ in the
PA-CndG method in the previous subsection, the vector $p_k$ in the
PDA-CndG method is defined as a weighted average of $f'(z_{i-1})$, $i = 1, \ldots, k$,
for some properly chosen weights $\theta_i$, $i = 1, \ldots, k$. 
This algorithm can also be viewed as the projection-free version of an $\infty$-memory variant of Nesterov's 
accelerated gradient method as stated in \cite{Nest05-1,tseng08-1}.

Note that by convexity of $f$, the function $\Psi_k(x)$ given by
\beq \label{def_psi}
\Psi_k(x) := 
 \left\{
\begin{array}{ll}
0, & k = 0, \\
\Theta_k^{-1} \sum_{i=1}^k \theta_i l_f(z_{i-1}; x), & k \ge 1, 
\end{array}
\right.
\eeq
underestimates $f(x)$ for any $x \in X$. In particular,
by the definition of $x_k$ in Algorithm~\ref{algPDACG}, we have
\beq \label{PD_LB}
\Psi_k(x_k) \le \Psi_k(x) \le f(x), \ \ \forall \, x \in X,
\eeq
and hence $\Psi_k(x_k)$ provides a lower bound on the optimal value $f^*$
of problem \eqnok{cp}. In order to establish the convergence of the PDA-CndG method, 
we first need to show a simple technical result about $\Psi_k(x_k)$.

\begin{lemma} \label{PD_tech}
Let $\{x_k\}$ and $\{z_k\}$ be the two sequences computed by the PDA-CndG method.
We have
\beq \label{bnd_psi}
\theta_k \, l_f(z_{k-1}; x_k) \le \Theta_k \Psi_k(x_k) - \Theta_{k-1} \Psi_{k-1}(x_{k-1}), \ \ k = 1, 2, \ldots,
\eeq
where $l_f(\cdot\, ;\cdot)$ and $\Psi_k(\cdot)$ are defined in \eqnok{def_lf}
and \eqnok{def_psi}, respectively.
\end{lemma}

\begin{proof}
It can be easily seen from \eqnok{def_psi} and the definition of $x_k$ in Algorithm~\ref{algPDACG} 
that $x_k \in \Argmin_{x \in X} \Psi_k(x)$ and hence that $\Psi_{k-1}(x_{k-1}) \le \Psi_{k-1}(x_k)$. 
Using the previous observation and \eqnok{def_psi}, we
obtain
\beqas
\Theta_k \Psi_k(x_k) &=& \sum_{i=1}^k \theta_i l_f(z_{i-1}; x_i) 
= \theta_k \, l_f(z_{k-1}; x_k) + \sum_{i=1}^{k-1} \theta_i l_f(z_{i-1}; x_i)\\
&=& \theta_k l_f(z_{k-1}; x_k) + \Theta_{k-1} \Psi_{k-1}(x_k) \\
&\ge& \theta_k l_f(z_{k-1}; x_k) + \Theta_{k-1} \Psi_{k-1}(x_{k-1}).
\eeqas
\end{proof}

We are now ready to establish the main convergence properties of the PDA-CndG method.

\begin{theorem} \label{Theorem_PDACG}
Let $\{x_k\}$ and $\{y_k\}$ be the two sequences generated by the PDA-CndG method applied to
problem~\eqnok{cp} with the stepsize policy in \eqnok{FW_step1} or \eqnok{FW_step2}.
Also let $\{\gamma_k\}$ be defined in \eqnok{def_Gamma}. 
If the parameters $\theta_k$ are chosen such that
\beq \label{def_nu}
\theta_k \Theta_k^{-1} = \gamma_k, \ \ k = 1, 2, \ldots,
\eeq
Then, we have
\beq \label{Convergence_PDACG}
f(y_k) - f^* \le f(y_k) - \Psi_k(x_k) \le \frac{2 L}{k(k+1)} \sum_{i=1}^k \|x_i - x_{i-1}\|^2 \le \frac{2 L D_X^2}{k+1}
\eeq
for any $k = 1, 2, \ldots$, where $L$ is given by \eqnok{smoothness}.
\end{theorem}

\begin{proof}
Denote $\tilde y_k = (1-\gamma_k) y_{k-1} + \gamma_k x_k$. 
It follows from \eqnok{FW_step1} (or \eqnok{FW_step2}) and the definition of $y_k$ that
$f(y_k) \le f(\tilde y_k)$. Also noting that, by definition, we have
$z_{k-1} = (1-\gamma_k) y_{k-1} + \gamma_k x_{k-1}$ and hence 
\[
\tilde y_k - z_{k-1} = \gamma_k (x_k - x_{k-1}).
\]
Using these two observations, \eqnok{smoothness}, the definitions of $x_k$ in Algorithm~\ref{algPDACG},
 the convexity of $f$ and \eqnok{bnd_psi}, we obtain
\beqa
f(y_k) &\le& f(\tilde y_k) \le 
l_f(z_{k-1}; \tilde y_k) + \frac{L}{2} \|\tilde y_k - z_{k-1}\|^2 \nn\\
&=& (1 - \gamma_k) l_f(z_{k-1}; y_{k-1})  + \,\, \gamma_k l_f(z_{k-1}; x_k) 
+ \frac{L}{2} \gamma_k^2 \|x_k - x_{k-1}\|^2 \nn \\
&=& (1-\gamma_k) f(y_{k-1}) + \gamma_k l_f(x_k; z_{k-1}) + \frac{L}{2} \gamma_k^2 \|x_k - x_{k-1}\|^2 \nn \\
&\le& (1-\gamma_k) f(y_{k-1}) + \gamma_k \theta_k^{-1} [\Theta_k \Psi_k(x_k) - \Theta_{k-1} \Psi_{k-1}(x_{k-1})] + \frac{L}{2} \gamma_k^2 \|x_k - x_{k-1}\|^2 .
\eeqa
Also, using \eqnok{def_nu} and the fact that $\Theta_{k-1} = \Theta_k- \theta_k$,
we have 
\beqas
\gamma_k \theta_k^{-1} [\Theta_k \Psi_k(x_k) - \Theta_{k-1} \Psi_{k-1}(x_{k-1})] &=& \Psi_k(x_k) - \Theta_{k-1}\Theta_k^{-1}\Psi_{k-1}(x_{k-1})\\
&=& \Psi_k(x_k) - \left(1 - \theta_k \Theta_k^{-1} \right) \Psi_{k-1}(x_{k-1})\\
&=& \Psi_k(x_k) - (1-\gamma_k) \Psi_{k-1}(x_{k-1}).
\eeqas
Combining the above two relations and re-arranging the terms, we obtain
\[
f(y_k) - \Psi_k(x_k) \le (1-\gamma_k) \left[f(y_{k-1}) - \Psi_{k-1}(x_{k-1})\right] 
+ \frac{L}{2} \gamma_k^2 \|x_k - x_{k-1}\|^2,
\]
which, in view of Lemma~\ref{tech_result}, \eqnok{def_Gamma} and \eqnok{def_Gamma1},
then implies that
\[
f(y_k) - \Psi_k(x_k) \le \frac{2 L}{k(k+1)} \sum_{i=1}^k \|x_i - x_{i-1}\|^2.
\]
Our result then immediately follows from \eqnok{PD_LB} and the above inequality.
\end{proof}

\vgap

We now add a few remarks about the results obtained in Theorem~\ref{Theorem_PDACG}.
Firstly, observe that we can simply set $\theta_k = k$, $k = 1, 2, \ldots$
in order to satisfy \eqnok{def_nu}. Secondly, in view of the
discussion after Theorem~\ref{Theorem_PACG}, the PDA-CndG method is also an optimal LCP method
for ${\cal F}^{1,1}_{L,\|\cdot\|}(X)$ when $n$ is sufficiently large, since
its the rate of convergence is exactly the same as the one for the PA-CndG method. 
In addition, its rate of convergence is invariant of the selection of 
the norm $\|\cdot\|$ (see \eqnok{norm_inv}). 
Thirdly, according to \eqnok{Convergence_PDACG},
we can compute an online lower bound $\Psi_k(x_k)$ on the optimal value $f^*$,
and terminate the PDA-CndG method based on the optimality gap $f(y_k) - \Psi_k(x_k)$.

\vgap

Similar to the PA-CndG method, the rate of convergence of the PDA-CndG method
depends on $x_k - x_{k-1}$, which in turn depends on the
geometry of $X$ and the input vectors $p_k$ and $p_{k-1}$
to the $\LMO$ oracle. 
One can easily check 
the closeness between $p_k$ and $p_{k-1}$. Indeed, by the definition of $p_k$, we 
have $p_k = \Theta_k^{-1} [(1 - \theta_k) p_{k-1} + \theta_k f_k'(z_{k-1})$ and hence
\beq \label{PD_close}
p_k - p_{k-1} = \Theta_k^{-1} \theta_k [p_{k-1} + f'_k(z_{k-1})] = \gamma_k [p_{k-1} + f'_k(z_{k-1})],
\eeq
where the last inequality follows from \eqnok{def_nu}. 
Noting that by \eqnok{smoothness}, we have $\|f'(x)\|_* \le \|f'(x^*)\|_* + L D_X$ for any $x \in X$
and hence that $\|p_k\|_* \le \|f'(x^*)\|_* + L D_X$ due to the definition of $p_k$. Using these
observations, we obtain 
\beq \label{PD_close1}
\|p_k - p_{k-1}\|_* \le 2 \gamma_k [\|f'(x_*)\|_* + L D_X] , \ \  k \ge 1. 
\eeq
Hence, under certain continuity assumptions on the $\LMO$ oracle, we can obtain a result
similar to Corollary~\ref{cor_PACG}. Note that both stepsize policies in
\eqnok{FW_step1} and \eqnok{FW_step2} can be used in this result.

\begin{corollary} \label{cor_PDACG}
Let $\{y_k\}$ be the sequences generated by the PDA-CndG method applied to
problem~\eqnok{cp} with the stepsize policy in \eqnok{FW_step1} or \eqnok{FW_step2}.
Assume that \eqnok{def_nu} holds. Also suppose that the $\LMO$ oracle satisfies \eqnok{Lipschitz_LO}
for some $\rho \in (0,1]$ and $Q > 0$.
Then we have, for any $k \ge 1$,
\beq \label{Convergence_PDACG1}
f(y_k) - f^* \le
 {\cal O}(1) \left\{
\begin{array}{ll}
L Q^2 \left[ \|f'(x_*)\|_* + L\ D_{X}\right]^{2 \rho} / \left[(1-2 \rho) \, k^{2 \rho +1}\right], & \rho \in (0, 0.5),\\
L Q^2 \left[ \|f'(x_*)\|_* + L\ D_{X}\right] \log (k+1)/ k^{2}, & \rho = 0.5,\\
L Q^2 \left[ \|f'(x_*)\|_* + L\ D_{X}\right]^{2 \rho}/ \left[(2 \rho - 1) \, k^2\right], &\rho \in (0.5,1].
\end{array}
\right.
\eeq
\end{corollary}

Similar to Corollary~\ref{cor_PACG}, Corollary~\ref{cor_PDACG} also helps to build some connections between
LCP methods and the more general optimal first-order method. However, these results 
are more of theoretical interest only, since the $\LMO$ oracle 
does not necessarily satisfy \eqnok{Lipschitz_LO} for any $\rho > 0$, but only for $\rho = 0$ and
$Q = D_X$.

\vgap


\subsection{Numerical Illustration} \label{sec_num}
Our goal in this subsection is to compare through some preliminary numerical experiments
the three LCP methods for solving ${\cal F}^{1,1}_{L,\|\cdot\|}(X)$, i.e., CndG, PA-CndG and PDA-CndG,
all of which share similar worst-case complexity bounds.
More specifically, we conduct three sets of experiments for solving quadratic programming (QP) problems 
over a few different types of feasible sets. In our first set of experiments, we consider the QP problems over a standard simplex
or spectrahedron. In particular, let $A \in \bbr^{m \times n}$,
${\cal A}: \bbr^{n \times n} \to \bbr^m$ and $b\in \bbr^m$ be given, the QP over a standard
simplex and over a standard spectrahedron, respectively, are defined as $\min_{x \in \Delta_n} \|Ax - b\|_2^2$
and $\min_{x \in S_n} \|{\cal A} x - b\|^2_2$, where   
\beq \label{QP1}
\Delta_n := \left\{x \in \bbr^n: \sum_{i=1}^n x_i = 1, x_i \ge 0, i = 1, \ldots, n\}\right\}
\eeq
and
\beq \label{QP2}
S_n := \left\{x \in \bbr^{n\times n}: \Tr(x) = 1,  x \succeq 0 \right\}.
\eeq
We can easily see that $\Delta_n \subset S_n$ by setting $x$ to be diagonal in \eqnok{QP2}.
In our second set of experiments, we consider the QP problems over a
hypercube, i.e., $\min_{x \in B_n} \|Ax - b\|_2^2$, where
\beq \label{QP3}
B_n := \left\{x \in \bbr^{n}: x_i \in [0,1], i = 1, \ldots, n \right\}.
\eeq
It is well-known that to solve these problems becomes more and more difficult as $n$ increases.
In our last set of experiments, we consider the QP problems over a
hypercube intersected with a simplex, i.e.,
$\min_{x \in H_n(r)} \|Ax - b\|_2^2$, where
\beq \label{QP4}
H_n(r) := \left\{x \in \bbr^n: \sum_{i=1}^n x_i \le r n, x_i \in [0,1] \right\}
\eeq
for some $r \in (0,1]$. These problems arise from certain important applications, including
compressed sensing and portfolio optimization. 

Our experiments have been carried out on a set of instances 
that are randomly generated as follows. Firstly, we randomly generate a feasible solution $s_0$ in $\Delta_n$, $B_n$,
$S_n$ and $H_n(r)$, and a linear operator $A: \bbr^n \to \bbr^m$ (or ${\cal A}: \bbr^{n\times n} \to \bbr^m$) with sparse
entries uniformly distributed over $[0,1]$. Then we compute $b = A s_0$ (or $b = {\cal A} s_0$).
Clearly in this way, the optimal values of these instances are given by $0$. Also note that
a sparsity parameter $d$ has been used when generating the linear operator $A$ (or ${\cal A}$).
Totally $36$ instances have been generated, see Table~\ref{tab_inst}
for more details.

The CndG, PA-CndG and PDA-CndG algorithms are implemented in Matlab R2011b. 
Observe that we have discussed two stepsize policies for these
algorithms: the stepsize policy \eqnok{FW_step1} is
the one that we have used in our experiments for its simplicity,
while one can also use the stepsize policy \eqnok{FW_step2} with more expensive iteration costs. 
The parameter $\theta_k$ in PDA-CndG is simply set to $\theta_k = k$.
The initial point $y_0$ is randomly generated and remains the same for different algorithms.
We report the results in Tables~\ref{result_simplex},
\ref{result_CUBE1} and \ref{result_CUBE2}, respectively, for minimization over simplex/spectrahedron, hypercube,
and hypercube intersected with simplex. For each problem instance, we compute the objective values at
the search points $y_0$, $y_{100}$ and $y_{1000}$, and the total CPU time 
(in seconds, Intel Core i7-2600 3.4 GHz) required for performing
$1,000$ iterations of these algorithms. 


\begin{table}
\caption{Randomly generated instances}
\vgap

\centering
\label{tab_inst}
\footnotesize
\begin{tabular}{|c|c|c|c|c||c|c|c|c|c|}
\hline
Inst. & Domain & $n$ & $m$ & $d$ & Inst. & Domain & $n$ & $m$ & $d$\\
\hline
SIM11 & $\Delta_n$ & $2,000$ & $500$ & $1.0$ & SIM12 & $\Delta_n$ & $2,000$ & $1,000$ & $1.0$\\ 
SIM21 & $\Delta_n$ & $4,000$ & $1,000$ & $0.8$ & SIM22 & $\Delta_n$ & $4,000$ & $2,000$ & $0.8$\\  
SIM31 & $\Delta_n$ & $8,000$ & $2,000$ & $0.6$ & SIM32 & $\Delta_n$ & $8,000$ & $4,000$ & $0.6$\\
SPE41 & $S_n$ & $100$ & $500$ & $0.6$ & SPE42 & $S_n$ & $100$ & $1,000$ & $0.6$ \\
SPE51 & $S_n$ & $200$ & $500$ & $0.4$ & SPE52 & $S_n$ & $200$ & $1,000$ & $0.4$ \\
SPE61 & $S_n$ & $400$ & $500$ & $0.2$ & SPE62 & $S_n$ & $400$ & $1,000$ & $0.2$ \\
\hline
CUB11 & $C_n$ & $500$ & $100$ & $1.0$ & CUB12 & $C_n$ & $500$ & $200$ & $1.0$ \\ 
CUB21 & $C_n$ & $1,000$ & $250$ & $1.0$ & CUB22 & $C_n$ & $1,000$ & $5,00$ & $1.0$ \\ 
CUB31 & $C_n$ & $2,000$ & $500$ & $1.0$ & CUB32 & $C_n$ & $2,000$ & $1,000$ & $1.0$ \\
CUB41 & $C_n$ & $4,000$ & $1,000$ & $0.8$ & CUB42 & $C_n$ & $4,000$ & $2,000$ & $0.8$ \\ 
CUB51 & $C_n$ & $8,000$ & $2,000$ & $0.6$ & CUB52 & $C_n$ & $8,000$ & $4,000$ & $0.6$ \\
CUB61 & $C_n$ & $16,000$& $4,000$ & $0.4$ & CUB62 & $C_n$ & $16,000$& $8,000$ & $0.4$\\
\hline
HYB11 & $H_n(0.25)$ & $4,000$ & $1,000$ & $0.8$ & HYB12 & $H_n(0.25)$ & $4,000$ & $2,000$ & $0.8$ \\
HYB21 & $H_n(0.5)$ & $4,000$ & $1,000$ & $0.8$ & HYB22 & $H_n(0.5)$ & $4,000$ & $2,000$ & $0.8$ \\
HYB31 & $H_n(0.25)$ & $8,000$ & $2,000$ & $0.6$ & HYB32 & $H_n(0.25)$ & $8,000$ & $4,000$ & $0.6$ \\
HYB41 & $H_n(0.5)$ & $8,000$ & $2,000$ & $0.6$ & HYB42 & $H_n(0.5)$ & $8,000$ & $4,000$ & $0.6$ \\
HYB51 & $H_n(0.25)$ & $16,000$& $4,000$ & $0.4$ & HYB52 & $H_n(0.25)$ & $16,000$& $8,000$ & $0.4$\\
HYB61 & $H_n(0.5)$ & $16,000$& $4,000$ & $0.4$ & HYB62 & $H_n(0.5)$ & $16,000$& $8,000$ & $0.4$\\
\hline
\end{tabular}
\end{table}

\begin{table}
\caption{Comparison of CndG methods for minimization over simplex/spectrahedron}
\vgap

\centering
\label{result_simplex}
\footnotesize
\begin{tabular}{|c|c|ccc|ccc|ccc|}
\hline
\multicolumn{1}{|c|}{} & \multicolumn{1}{|c|}{}& \multicolumn{3}{|c|}{CndG}& \multicolumn{3}{|c|}{PA-CndG} & \multicolumn{3}{|c|}{PDA-CndG}\\
\hline
Inst & $f(y_0)$ & $f(y_{100})$ & $f(y_{1000})$ & Time & $f(y_{100})$ & $f(y_{1000})$ & Time & $f(y_{100})$ & $f(y_{1000})$ & Time \\
\hline
SIM11 & 1.27e-1 & 1.09e-1 & 2.81e-3 & 1.84 & 1.14e-1 & 2.69e-3 & 3.54 & 1.27e-1 & 5.02e-3& 3.32 \\ 
SIM12 & 2.49e-1 & 2.49e-1 & 8.20e-3 & 3.47 & 2.91e-1 & 9.22e-3 & 6.75 & 2.49e-1 & 1.30e-2& 6.85 \\
SIM21 & 1.25e-1 & 1.25e-1 & 7.94e-3 & 6.15 & 1.25e-1 & 7.87e-3 & 12.01 & 1.25e-1 & 1.49e-2& 12.33 \\
SIM22 & 2.53e-1 & 2.53e-1 & 2.55e-2 & 11.99 & 2.53e-1 & 2.54e-2 & 23.17 & 2.53e-1 & 3.74e-2& 23.19\\
SIM31 & 1.17e-1 & 1.17e-1 & 2.32e-2 & 18.75 & 1.17e-1 & 2.36e-2 & 37.24 & 1.17e-1 & 4.13e-2& 37.18\\
SIM32 & 2.25e-1 & 2.25e-1 & 7.02e-2 & 38.78 & 2.25e-1 & 6.80e-2 & 75.70 & 2.25e-1 & 9.87e-2& 75.87\\
SPE41 & 5.51e+1 & 1.65e-1 & 1.66e-3 & 14.00 & 2.80e-1 & 2.91e-3 &19.79 & 4.47e-1 & 5.72e-3& 21.21\\
SPE42 & 1.01e+2 & 5.24e-1 & 6.33e-3 & 18.90 & 8.74e-1 & 1.13e-2 &30.79& 9.70e-1 & 1.77e-2&33.77\\
SPE51 & 1.52e+1 & 8.85e-2 & 8.40e-4 & 30.90 & 1.90e-1 & 1.65e-3 &46.65& 1.87e-1 & 1.94e-3&48.33\\
SPE52 & 3.65e+1 & 1.97e-1 & 1.99e-3 & 45.89 & 3.22e-1 & 3.55e-3 &78.00& 8.73e-1 & 1.37e-2&79.86\\
SPE61 & 3.01e+0 & 3.90e-2 & 3.84e-4 & 73.80 & 1.00e-1 & 9.60e-4 &104.97& 5.23e-2 & 5.64e-4&112.80\\
SPE62 & 5.92e+0 & 8.02e-2 & 8.00e-4 & 109.34 & 1.58e-1 & 1.44e-3 &177.80& 1.82e-1 & 2.02e-3&181.01\\
\hline
\end{tabular}
\end{table}

\begin{table}
\caption{Comparison of CndG methods for minimization over hypercube}
\vgap

\centering
\label{result_CUBE1}
\footnotesize
\begin{tabular}{|c|c|ccc|ccc|ccc|}
\hline
\multicolumn{1}{|c|}{} & \multicolumn{1}{|c|}{}& \multicolumn{3}{|c|}{CndG}& \multicolumn{3}{|c|}{PA-CndG} & \multicolumn{3}{|c|}{PDA-CndG}\\
\hline
Inst & $f(y_0)$ & $f(y_{100})$ & $f(y_{1000})$ & Time & $f(y_{100})$ & $f(y_{1000})$ & Time & $f(y_{100})$ & $f(y_{1000})$ & Time \\
\hline
CUB11 & 1.98e+5 & 3.52e+1 & 3.50e-1 & 0.13 & 2.04e+1 & 1.41e+0 & 0.23 & 2.94e+0 & 3.17e-2& 0.24 \\ 
CUB12 & 4.02e+4 & 9.96e+1 & 3.64e+0 & 0.18 & 1.16e+2 & 1.08e+1 & 0.36 & 4.81e+0 & 1.65e-2& 0.34 \\
CUB21 & 2.18e+6 & 5.23e+2 & 1.53e+0 & 0.40 & 4.33e+2 & 4.46e+1 & 0.70 & 1.51e+1 & 3.24e-1& 0.72 \\
CUB22 & 4.49e+6 & 9.61e+2 & 7.60e+1 & 0.84 & 8.90e+2 & 1.86e+2 & 1.45 & 6.67e+1 & 1.67e-1& 1.47\\
CUB31 & 1.73e+7 & 2.43e+3 & 2.13e+2 & 1.82 & 2.23e+3 & 4.68e+2 & 3.33 & 2.60e+2 & 1.67e+0& 3.33\\
CUB32 & 3.25e+7 & 5.35e+3 & 6.74e+2 & 3.43 & 5.85e+3 & 1.56e+3 & 6.57 & 7.84e+2 & 1.41e+0& 6.60\\
CUB41 & 1.03e+8 & 1.03e+4 & 1.38e+3 & 6.02 & 9.50e+3 & 2.46e+3 &11.95 & 1.58e+3 & 1.23e+1& 11.80\\
CUB42 & 1.95e+8 & 2.24e+4 & 4.64e+3 & 11.82 & 2.00e+4 & 8.88e+3 &23.05& 4.95e+3 & 1.04e+1&23.74\\
CUB51 & 5.43e+8 & 4.65e+4 & 9.83e+3 & 18.69 & 4.70e+4 & 1.22e+4 &37.22& 8.70e+3 & 6.63e+1&37.44\\
CUB52 & 1.09e+9 & 7.38e+4 & 2.74e+4 & 38.62 & 7.60e+4 & 3.48e+4 &75.88& 2.21e+4 & 5.53e+1&76.18\\
CUB61 & 2.32e+9 & 1.39e+5 & 4.56e+4 & 59.04 & 1.13e+5 & 4.94e+4 &117.40& 2.96e+4& 3.60e+2&116.52\\
CUB62 & 4.60e+9 & 2.26e+5 & 1.25e+5 & 115.62& 2.71e+5 & 1.39e+5 &228.64& 9.79e+4& 2.35e+2&226.36\\
\hline
\end{tabular}
\end{table}

\begin{table}
\caption{Comparison of CndG methods for minimization over hypercube intersected with simplex}
\vgap

\centering
\label{result_CUBE2}
\footnotesize
\begin{tabular}{|c|c|ccc|ccc|ccc|}
\hline
\multicolumn{1}{|c|}{} & \multicolumn{1}{|c|}{}& \multicolumn{3}{|c|}{CndG}& \multicolumn{3}{|c|}{PA-CndG} & \multicolumn{3}{|c|}{PDA-CndG}\\
\hline
Inst & $f(y_0)$ & $f(y_{100})$ & $f(y_{1000})$ & Time & $f(y_{100})$ & $f(y_{1000})$ & Time & $f(y_{100})$ & $f(y_{1000})$ & Time \\
\hline
HYB11 & 1.58e+7 & 1.12e+3 & 7.80e+1 & 6.60 & 1.14e+3 & 8.27e+1 & 12.50 & 4.88e+1 & 2.72e-1& 12.33 \\ 
HYB12 & 3.11e+7 & 3.56e+3 & 1.11e+3 & 12.39 & 3.45e+3 & 1.06e+3 & 23.91 & 1.12e+3 & 8.16e+0& 24.27 \\
HYB21 & 1.00e+8 & 2.48e+3 & 6.18e+2 & 6.58 & 2.10e+3 & 7.50e+2 & 12.23 & 4.66e+2 & 1.05e+1& 12.32 \\
HYB22 & 2.00e+8 & 7.73e+3 & 3.39e+3 & 12.32 & 6.61e+3 & 3.82e+3 & 23.97 & 1.53e+3 & 7.67e+0& 23.83\\
HYB31 & 8.45e+7 & 5.75e+3 & 3.67e+2 & 20.06 & 4.73e+3 & 3.85e+3 & 39.09 & 3.65e+2 & 1.77e+0& 38.84\\
HYB32 & 1.67e+8 & 1.58e+4 & 4.29e+3 & 40.27 & 1.45e+4 & 4.41e+3 & 78.10 & 5.45e+3 & 3.94e+1& 78.26\\
HYB41 & 5.47e+8 & 1.02e+4 & 3.04e+3 & 20.01 & 1.00e+4 & 3.44e+3 &38.15 & 3.95e+3 & 5.25e+1& 38.43\\
HYB42 & 1.06e+9 & 3.09e+4 & 1.56e+4 & 39.90 & 3.15e+4 & 1.73e+4 &79.91& 1.04e+4 & 4.35e+1&79.07\\
HYB51 & 3.57e+8 & 1.82e+4 & 1.82e+3 & 60.14 & 1.74e+4 & 1.81e+3 &117.99& 1.55e+3 & 7.00e+0&117.76\\
HYB52 & 7.10e+8 & 5.58e+4 & 1.62e+4 & 117.23 & 5.52e+4 & 1.71e+4 &231.05& 1.66e+4 & 1.34e+2&232.59\\
HYB61 & 2.33e+9 & 3.88e+4 & 1.26e+4 & 60.64 & 3.76e+4 & 1.42e+4 &119.09& 1.84e+4& 1.98e+2&118.71\\
HYB62 & 4.69e+9 & 1.08e+5 & 5.34e+4 & 117.80& 1.00e+5 & 5.93e+4 &233.31& 6.85e+4& 2.02e+2&232.12\\
\hline
\end{tabular}
\end{table}

We make a few observations about the results obtained in Tables~\ref{result_simplex}, \ref{result_CUBE1} and
\ref{result_CUBE2}. Firstly, for solving the QP problems over a standard simplex/spectrahedron,
all these three algorithms are about the same, with the CndG method slightly outperforming the other two.
Secondly, for solving the QP problems over a hypercube, PDA-CndG can significantly outperform 
both CndG and PA-CndG by orders of magnitude. 
More specifically, as it can be seen from Table~\ref{result_CUBE1}, although
the CPU times for PDA-CndG are about as twice as the ones for CndG, the function values 
computed at the $100$ iterations of PDA-CndG are already comparable to those computed at the $1,000$
iterations for both CndG and PA-CndG. Moreover, the objective values at the $1,000$ iterations of the
PDA-CndG are better than those for CndG and PA-CndG by $1-3$ accuracy digits, and the difference
seems to become larger as $n$ increases. Thirdly, it can be seen from Table~\ref{result_CUBE2}
that PDA-CndG also outperforms both CndG and PA-CndG by up to $2$ orders of magnitude for solving
the QP problems over a hypercube intersected with simplex. Therefore, we conclude that 
the PDA-CndG method, although sharing similar worst-case complexity bounds with both CndG and PA-CndG, 
might significantly outperform the latter two algorithms for solving certain
classes of CP problems, e.g., those with box-type constraints.

\section{Concluding remarks}
In this paper, we study a new class of optimization algorithms, namely the LCP methods,
which covers the classic CndG method as a special case. We establish a few lower complexity
bounds for these algorithms to solve different classes of CP problems.
We formally show that the classic CndG method is an optimal LCP method for solving smooth
CP problems and present new variants of this algorithm that are optimal or nearly optimal
for solving certain saddle point and general nonsmooth problems under 
an $\LMO$ oracle. Finally, we develop a few new LCP methods, namely PA-CndG and PDA-CndG,
by properly modifying Nesterov's accelerated gradient method, and show that
they also exhibit the optimal rate of convergence for solving smooth CP problems 
under an $\LMO$ oracle. In addition, we demonstrate through our preliminary numerical experiments 
that the PDA-CndG method can significantly outperform the classic CndG for solving certain classes of 
large-scale CP problems.

\section*{Acknowledgement}
The author would like to thank  Martin Jaggi and Elad Hazan
for their help with improving the exposition of the paper and pointing out quite a few missing
references in the original version of the paper.
The author would also like to acknowledge support 
from the NSF grant CMMI-1000347, DMS-1319050, 
    ONR grant N00014-13-1-0036 and NSF CAREER Award CMMI-1254446. 

\bibliographystyle{plain}
\bibliography{../glan-bib}

\end{document}